\documentclass[11pt]{amsart}

\usepackage{times}
\usepackage{lineno} 
\usepackage{multicol}

\usepackage{amsthm,amssymb,mathrsfs,pstricks,booktabs}
\usepackage{mathtools,amsmath,fancyhdr,appendix,nccmath}
\usepackage[all]{xy} 

\usepackage{enumerate}

\usepackage{array}
\usepackage[left=2cm,top=2cm,right=2cm,bottom=2cm]{geometry}

\usepackage{xcolor}
\usepackage{verbatim}
\definecolor{ao}{rgb}{0.0, 0.5, 0.0}

\usepackage{hyperref}
\hypersetup{
   colorlinks=true,
   linkcolor=red,
    urlcolor=red,
   citecolor=green,
    hyperindex=true,
    bookmarks=true,
    pagebackref=true
}

\theoremstyle{plain}
\newtheorem{theorem}{Theorem}[section]

\newtheorem{corollary}[theorem]{Corollary}
\newtheorem{proposition}[theorem]{Proposition}

\theoremstyle{definition}
\newtheorem{definition}[theorem]{Definition}
\newtheorem{example}[theorem]{Example}

\newtheorem{remark}[theorem]{Remark}
\theoremstyle{remark}
\makeatletter
\renewenvironment{proof}[1][\proofname]{\par
  \pushQED{\qed}%
  \normalfont \topsep6\p@\@plus6\p@\relax
  \trivlist
  \item[\hskip\labelsep
    \textbf{#1}\@addpunct{:}]\ignorespaces
}{%
  \popQED\endtrivlist\@endpefalse
}
\makeatother

\newcommand\function[5]{%
  \begingroup
  \setlength\arraycolsep{0pt}
  #1\colon\begin{array}[t]{c >{{}}c<{{}} c}
             #2 & \to & #3 \\ #4 & \mapsto & #5 
          \end{array}%
  \endgroup}

\renewcommand{\emptyset}{\varnothing}
\newcommand{\p}{\mathfrak{p}}
\newcommand{\cl}{{\bf\mathrm{cl}}}
\renewcommand{\a}{\mathfrak{a}}
\newcommand{\m}{\mathfrak{m}}
\newcommand{\spec}{\,\mathrm{Spec}}
\newcommand{\spic}{\mathbf{Pic}}
\newcommand{\sC}{\mathbf{Cart}}
\newcommand{\Z}{\mathbb{Z}}
\newcommand{\ev}{_{\overline{0}}}
\newcommand{\od}{_{\overline{1}}}
\renewcommand{\ker}{\mathrm{Ker}}

\newcommand{\id}{\mathrm{id}}
\usepackage[mathscr]{euscript}
\renewcommand{\qed}{\hfill$\square$}
\renewcommand{\to}
{\longrightarrow}

\newcommand{\X}{\mathfrak{X}}

\renewcommand{\mapsto}{\longmapsto}

\renewcommand{\O}{\mathcal{O}}
\renewcommand{\hom}{\mathrm{Hom}}
\usepackage{dsfont}
\newcommand{\homm}{\underline{\mathrm{Hom}}}

\newcommand{\srings}{\mathsf{sRngs}}

\newcommand{\ann}{\mathrm{Ann}}
\newcommand{\J}{\mathfrak{J}}
\newcommand{\kdim}{\mathrm{Kdim}}
\newcommand{\ksdim}{\mathrm{Ksdim}}
\newcommand{\nil}{\mathrm{Nil}}
\newcommand{\sdim}{\mathrm{sdim}}

\newcommand{\ass}{\mathfrak{ass}}
\newcommand{\q}{\mathfrak{q}}

\renewcommand{\b}{\mathfrak{b}}

\usepackage{setspace}
\newcommand{\K}{\Bbbk}
 
\begin{document}

\title{Dedekind superrings  and related concepts}  

\author{Pedro  Rizzo}

\address{Instituto de Matemáticas, FCEyN, Universidad de Antioquia, 50010  Medellín, Colombia}

\email{pedro.hernandez@udea.edu.co}

\author{Joel Torres del Valle} 

\address{Instituto de Matemáticas, FCEyN, Universidad de Antioquia, 50010  Medellín, Colombia}

\email{joel.torres@udea.edu.co}

\author{Alexander Torres-Gomez}

\address{Instituto de Matemáticas, FCEyN, Universidad de Antioquia, 50010  Medellín, Colombia}
 
\email{galexander.torres@udea.edu.co}


\maketitle

\begin{abstract} This article investigates the properties of Dedekind superrings, invertible supermodules and projective supermodules within the $\mathbb{Z}_2$-graded framework. Rather than treating these entities as specialized instances of general noncommutative ring theory, we develop them intrinsically within the category of supercommutative superrings. We examine the structural parallels to the classical commutative framework and, more importantly, characterize the fundamental discrepancies that emerge in the $\mathbb Z_2$-graded setting. In particular, we show that many hallmark equivalences of classical Dedekind domains--including those involving integral closedness and the coincidence of principal and unique factorization domains--fail to persist in the presence of an odd part.
\end{abstract}


\section{Introduction} 

The theory of superrings and superalgebras has garnered significant research attention from both the physics and mathematics communities. By extending the classical commutative framework through a $\mathbb{Z}_2$-grading, these structures incorporate both commuting and anticommuting elements, effectively broadening the scope of algebraic study. Formally, a superalgebra is defined as an algebra over a field (or a commutative ring) equipped with a $\mathbb{Z}_2$-grading, while a superring is a $\mathbb{Z}_2$-graded ring whose homogeneous elements satisfy a supercommutativity relation.

The investigation of superrings integrates techniques from abstract algebra, category theory, and representation theory. In this graded setting, the structural properties of ideals and supermodules acquire new dimensions, offering deeper insights into the behavior of the odd (anticommuting) sector. While noncommutative rings generally present significant structural challenges, superrings offer a more tractable middle ground; the ``well-behaved" nature of their homogeneous elements allows for the systematic extension of commutative results to a noncommutative environment. Consequently, superring theory serves as a rigorous bridge for developing analogous notions in general noncommutative ring theory.

Motivated by the requirement to formally define dimension in supergeometry, Masuoka and Zubkov \cite{MASUOKA2020106245} extended the classical Krull dimension to the $\mathbb{Z}_2$-graded setting by introducing the Krull superdimension. This advancement was pivotal for defining regularity and smoothness for superschemes. Given the fundamental role of Dedekind domains in algebraic number theory and the theory of non-singular curves, it is natural to ask whether a robust theory of Dedekind superrings can be formulated.

The pursuit of noncommutative analogues of Dedekind domains has a rich history, particularly since the 1980s. Pioneering works \cite{robson, robsonet} (see also \cite[Chapter 5]{robsonbook}) explored these structures via fractional ideals and generalized notions of invertibility, while others have addressed classification problems in the general \cite{good} and graded cases \cite{vander}. To establish a rigorous framework for abstract, non-singular algebraic supercurves, the characterization of Dedekind superrings is essential. Following \cite{MASUOKA2020106245}, we define a Dedekind superring as a regular Noetherian superring of Krull superdimension $1\mid n$ (where $n \in \mathbb{N}$) and investigate the extent to which classical equivalences for Dedekind domains persist.

A central objective of this work is the development of the algebraic theory of invertible supermodules and fractional superideals. Remarkably, our analysis reveals that most standard characterizations of Dedekind domains fail in the supercontext. Specifically, the classical equivalences involving the invertibility of fractional ideals, unique factorization of ideals, and integral closedness do not hold; only the characterization via discrete valuation superrings remains valid under a suitable modification.

The notion of a discrete valuation superring diverges significantly in this setting. A regular local superring of even Krull superdimension $1$ does not necessarily arise from a valuation on its superfield of fractions. Indeed, while valuation superrings are inherently integrally closed \cite[Proposition 3.12]{RTT3}, we prove that non-trivial Dedekind superrings are never integrally closed (Remark \ref{remark:5.12} (ii)). We further examine two hallmark properties of classical Dedekind domains:
\begin{enumerate}
\item The property of being a Dedekind domain is preserved by integral closure in finite separable field extensions.

\item  The coincidence of Principal Ideal Domains and Unique Factorization Domains.
\end{enumerate}

Furthermore, our results on superrings with unique factorization \cite{RTT5} reveal sharp discrepancies with the classical theory. Ultimately, while certain features of Dedekind domains admit meaningful extensions, the super-theory manifests genuinely new and subtle phenomena that distinguish it from both the commutative and general noncommutative frameworks.

The paper is organized in the following manner. Section \ref{Sec:Basic:Def} provides a comprehensive overview of the fundamental concepts of superrings and supermodules that will be utilized throughout the paper. This includes prime ideals associated with supermodules and the definition of Krull superdimension for superrings. While many of the results presented in this section posses proofs that are analogous to their counterparts in commutative algebra, we nonetheless include them due to their relatively scarcity within the existing literature on supercommutative algebra. Section \ref{Sec:Invertibibility} introduces the concepts of invertible supermodule and fractional superideal, investigates their fundamental properties, and constructs the Picard and Cartier groups of Noetherian superrings. Section \ref{Sec:Projective:Supermodules} focuses on the study of projective supermodules and their relationship to invertible supermodules. In view of the detailed treatment of this topic provided by Moyre et al. \cite{morye2022notes}, our presentation here is concise and solely focuses on the essential aspects relevant to the present work. Section \ref{Sec:Dedekind} defines and analyzes the properties of Dedekind superrings. Section \ref{comments} concludes by discussing potential avenues for future research and raises several open questions.



\section{Preliminaries}\label{Sec:Basic:Def}


In this section we introduce some basic notation and conventions, which will be used through the paper. Here, $\Z_2=\{\,\overline{0}, \overline{1}\,\}$ represents the group of integers modulo 2.  Any ring $R$ in this paper is assumed to be unitary and non-trivial (i.e., $1\neq 0$). 

\subsection{Superrings}

\begin{definition}
    A $\Z_2$-graded unitary ring $R=R\ev\oplus R\od$ is called a \emph{superring}. 
\end{definition} 

Let $R$ be a superring. We define the \textit{parity} of $x\in R$ as the integer 

\[
|x|:=\begin{cases}
    0& \text{ if }x\in R\ev,\\
    1& \text{ if }x\in R\od.
\end{cases}
\]

The set $h(R):=R\ev\cup R\od$ is called the \textit{homogeneous elements of $R$} and each $x\in h(R)$ is said to be  \textit{homogeneous}.  A homogeneous nonzero element $x$ is called \emph{even} if $x\in R\ev$ and \textit{odd} if $x\in R\od$.
\medskip

Any superring $R$ in this work will be \textit{supercommutative}, i.e., $R\ev$ is contained in the center of $R$ and every element of $R\od$ squares to zero. If 2 is a non-zerodivisor in $R$, this is equivalent to 

\begin{equation}\label{Supercommutativity}
    xy=(-1)^{|x||y|}yx,\quad \text{ for all }\quad x, y\in h(R).
\end{equation} 

\begin{definition}    
Let $R$ and $R'$ be superrings. A \textit{morphism} $\phi:R\to R'$ is a ring homomorphism \emph{preserving the parity}, that is, such that $\phi(R_i)\subseteq R'_i$ for all $i\in\mathbb{Z}_2$.

A morphism $\phi:R\to R'$ is an \textit{isomorphism} if it is bijective. In this case, we say that $R$ and $R'$ are \emph{isomorphic superrings}, which we denote by $R\cong R'$. 
\end{definition}
 
\begin{definition}
Let $R$ be a superring. 
\begin{enumerate}
\item[i)] A \textit{superideal} of $R$ is a $\Z_2$-graded ideal $\a$ of $R$, that is, it admits a decomposition $\a=(\a\cap R\ev)\oplus(\a\cap R\od)$.
\item[ii)] An ideal $\p$ of $R$ is \textit{prime} (resp., \textit{maximal}) if $R/\p$ is an integral domain (resp., a field).
\end{enumerate}
\end{definition}

\begin{remark}\label{rmk:prime-ideal} Let $R$ be a superring. 

\begin{enumerate}
\item[i)] If $\a$ is a superideal of $R$, we denote $\a_i:=\a\cap R_i$, for each $i\in\Z_2$.
\item[ii)]  Any superideal $\a$ of $R$ is a two-sided ideal.
\item[iii)] If $\a$ is a superideal of $R$, it is easy to show that the quotient ring $R/\a=(R/\a)\ev\oplus(R/\a)\od$ becomes a superring with $\Z_2$-grading given by $(R\ev/\a\ev)\oplus(R\od/\a\od)$.
\item[iv)] The definition of prime ideal is equivalent to requiring that if $x, y$ are elements (resp., homogeneous elements) in $R$ whose product is in $\p$, one of them is  in $\p$ (see \cite[Lemma 4.1.2]{westrathesis}). In other words, any prime ideal is completely prime.  
\item[v)] Any prime ideal of $R$ is a superideal of the form $\p=\p\ev\oplus R\od$, where $\p\ev$ is prime in $R\ev$. Similarly, any maximal ideal of $R$ is a superideal of the form $\m=\m\ev\oplus R\od$, where $\m\ev$ is a maximal ideal of $R\ev$ (see e.g., \cite[Lemma 4.1.9]{westrathesis}).
\end{enumerate}
\end{remark}

\begin{definition} Let $R$ be a superring. 
\begin{itemize}
    \item[i)] The ideal $\J_R=R\cdot R\od=R\od^2\oplus R\od$ is called the \textit{canonical superideal of $R$}.
    \item[ii)] The \textit{superreduced} of $R$ is the commutative ring $\overline{R}=R/\J_R\cong R\ev/R\od^2$. 
    \item[iii)] $R$ is called a \textit{superdomain} if $\overline{R}$ is a domain or, equivalently, $\J_R$ is a prime ideal. 
    
    \item[iv)] $R$ is called a \textit{superfield} if $\overline{R}$ is a field or, equivalently, $\J_R$ is a maximal ideal.
\end{itemize}
\end{definition}

\begin{definition} A superring $R$ is called \emph{Noetherian} if $R$ satisfies one (and hence all) of the following equivalent conditions (for a proof see \cite[Proposition 3.3.1]{westrathesis}):

    \begin{itemize}
        \item[i)] $R$ has the ascending chain condition on superideals.
        \item[ii)] Any superideal $\a$ of $R$ is \emph{finitely generated}, that is, there exist finitely many homogeneous elements $a_1,\ldots, a_n$ in $\a$  such that $\a=Ra_1+\cdots+Ra_n$.
        \item[iii)] Any nonempty collection of superideals of $R$ has a maximal element.
    \end{itemize}
\end{definition}

In analogy with the commutative  construction, the localization of a  superring is as follows. Let $R$ be a superring and $S\subseteq R\ev$ a multiplicative set. \emph{The localization $S^{-1}R$ of $R$ at $S$}, is defined as the superring 

\[
S^{-1}R:=(S^{-1}R\ev)\oplus(S^{-1}R\od),
\]

\noindent where $R\od$ is regarded as $R\ev$\,-module. As usual, an element in $S^{-1}R$ is denoted by $x/y$ or $y^{-1}x$, where $x\in R$ and $y\in S$. If $S:=R\ev-\p\ev$, where $\p$ is a prime ideal, $S^{-1}R$ is often written as $R_\p$. Further, every time  we write $S=R-\p$, where $\p$ is a prime ideal, it must be understood that $S=R\ev-\p\ev$. For example, $R-\J_R$ stands for $R\ev-R\od^2$, etc. If $S$ is the set of non-zerodivisors of $R$, $S^{-1}R$ is usually denoted by  $K(R)$ and is called \emph{the total superring of fractions of $R$}. If $R$ is a superdomain, then the localization $R_{\J_R}$ is a superfield, denoted by $Q(R)$ and called the \textit{superfield of fractions} of $R$. Let $\lambda_R^S: R\to S^{-1}R,\displaystyle  \,x\mapsto\frac{x}{1}$ be the \textit{localization morphism}. It is easy to check that the pair $(S^{-1}R, \lambda_R^S)$ satisfies the universal property of localization.

\begin{definition}
A superring $R$ is said to be \emph{local} if it has a unique maximal ideal $\mathfrak{m}$. We denote a such local superring by the pair $(R,\mathfrak{m})$. 
\end{definition} 

The following proposition is proved in \cite[Proposition 5.1.9]{westrathesis}.

\begin{proposition} Let $R$ be a superring and $\p$ a prime ideal of $R$. Then, $(R_\p, R_\p\p)$ is local superring, with $R_\p\p=\p_\p$. \qed 
\end{proposition}


\subsection{Supermodules}  


From now on, a ring $R$ is always meant to be a superring.

\begin{definition}
     A left $\Z_2$-graded $R$-module $M=M\ev\oplus M\od$ is called an $R$-\textit{supermodule}.
\end{definition} 

Let $M$ be an $R$-supermodule. The  \textit{set of homogeneous elements of $M$} is $h(M):=M\ev\cup M\od$, and the  \textit{parity} of $m\in h(M)$ is defined in the obvious way. If $\lambda$ is a left action of $R$ on $M$, we define a right action $\rho$ of $R$ on $M$ by 

\[
\rho(m, a)=(-1)^{|a||m|}\lambda(a, m),\quad \text{ for all } \quad x\in h(R)\text{ and }m\in h(M).
\]

\noindent It is easy to verify that these  actions commute. That is, $\lambda(a, \rho(m, a'))=\rho(\lambda(a, m), a')$ for all $a, a'\in R$ and $m\in M$. Thus, $M$ is both a right and a left $R$-supermodule. Hereafter any  supermodule is assumed to have these two compatible structures. 

\begin{definition}
     Let $M$ and $N$ be $R$-supermodules. A  \textit{morphism} $\phi:M\to N$ is a homomorphism of $R$-modules which  \textit{preserves the parity}, i.e., $\phi(M_{i})\subseteq N_{i}$, for all $i\in\Z_2$. 
\end{definition}

The set of morphisms $M\to N$ is denoted by $\hom_R(M, N)$. Let $\homm_R(M,N)$ be the set of all homomorphisms (of $R$-modules) from $M$ to $N$. This set is an $R$-supermodule with components   

\[
\homm_R(M,N)_{i}:=\left\{\psi\in \homm_R(M,N)\mid\,\psi(M_{j})\subseteq N_{i+j},\,\forall j\in\Z_2\right\}\text{ for all }i\in\Z_2,
\]

\noindent where the left action of $R$ is given by $(a\psi)(m)=a\psi(m)$ 
for all $a\in h(R)$ and $m\in h(M)$. It is easy to see that $\homm_R(M,N)\ev=\text{Hom}_R(M,N)$ (see \cite[Lemma 2.9]{morye2022notes} and \cite[Chapter 6]{westrathesis}).

\begin{definition}
    An $R$-supermodule $M$ is \emph{finitely generated} if as $R$-module is generated by finitely many homogeneous elements.
\end{definition}

Let $\text{sMod}_R$ be the category of $R$-supermodules. We denote by $\text{smod}_R$ the full subcategory of $\text{sMod}_R$ whose objects are supermodules finitely generated over $R$.
\medskip

An important endofunctor $\Pi:\text{sMod}_R\longrightarrow \text{sMod}_R$ (resp. $\Pi:\text{smod}_R\longrightarrow \text{smod}_R$) is the \textit{parity swapping functor} defined on objects as $M\mapsto\Pi M:=(\Pi M)\ev\oplus(\Pi M)\od$, where $(\Pi M)\ev=M\od$ and $(\Pi M)\od=M\ev$; on morphisms, $(\psi:M\to N)\mapsto (\Pi \psi:\Pi M\to \Pi N)$ which is defined by $\psi$ itself (\cite[\S\,6.1.2]{westrathesis}).

\begin{definition}
    A morphism $\psi:M\to N$ is called an \textit{isomorphism} if it is bijective. 
    
    We say that $M$ and $N$ are \emph{isomorphic} if there exists an isomorphism between them, and it is denoted by  $M\simeq N$.
\end{definition}

If $M$ and $N$ are $R$-supermodules, it can be proved that $M\oplus N$ and $M\otimes_R N$ are $R$-supermodules. Likewise, if $M$ and $N$ are $R$-subsupermodules of $P$, then $M+N$ and $M\cap N$ are $R$-subsupermodules of $P$. Moreover, if $N\subseteq M$ is an $R$-subsupermodule of $M$, then $M/N$ is an $R$-supermodule as well. For example, we define the \textit{superreduced} of $M$ as the $R$-supermodule $\overline{M}=M/\J_R M$, where $\J_R M$ is an example of an  $R$-subsupermodule  of $M$ of the type $\a M$, that is the \emph{product} of $M$ and the superideal $\a$ of $R$.

It is possible to show that the category $\text{sMod}_R$ is a strict monoidal (abelian) category with ``tensor'' product the usual tensor product ``$\otimes_R$" of supermodules whose unit is given by $R$. It is worth mentioning that the functor $\homm_R(M,-)$ is the left-adjoint of the functor $-\otimes_R M$ (see \cite[Chapter  1]{etingof} and \cite[Chapter 6]{westrathesis}).

Let $M$ be an $R$-supermodule and $S\subseteq R\ev$ a multiplicative set. The \emph{localization of $M$ at $S$} is  defined to be the $S^{-1}R$-supermodule

\[
S^{-1}M=(S^{-1}M\ev)\oplus(S^{-1}M\od).
\]

Each element in $S^{-1}M$ is denoted by $m/s$ or $s^{-1}m$, where $m\in M$ and $s\in S$. In the case $S=R-\p$, where $\p$ is a prime ideal of $R$, $S^{-1}M$ is denoted by $M_\p$.


Let $\lambda_M^S:M\to S^{-1}M, m\mapsto\displaystyle\frac{m}{1}$, be the \textit{localization morphism}. It can be easily proved that the pair $(S^{-1}M, \lambda_M^S)$ satisfies the universal property of localization. 

\begin{remark}\label{rem:prodK}
Let $R$ be a  superring. Observe that if $M$ and $N$ are $R$-subsupermodules of $K(R)$, then it makes sense to define \emph{the product of $M$ and $N$} as the $R$-subsupermodule of $K(R)$ given by 

\[
M\cdot N:=\left\{\sum_i m_in_i\mid\,m_i\in M,\,n_i\in N\right\}.
\]
 
\end{remark}

\begin{remark}
The localization functor $S^{-1}(\cdot):\text{sMod}_R\longrightarrow \text{sMod}_{S^{-1}R}$, is not only an exact functor but also a strict monoidal functor (see \cite[\S 2.4]{etingof}). Furthermore, the ``Local-Global Principle'' holds true in this context (see \cite[Propositions 3.9, 3.10]{atiyah}).  
\end{remark}


\subsection{Supervector spaces and superalgebras} 


\begin{definition}
    Let $\Bbbk$ be a field. A \textit{supervector space over} $\K$ is a $\K$-supermodule. If $V=V\ev\oplus V\od$ is a supervector space over $\K$, we define the \emph{superdimension} $\mathrm{sdim}_\K(V)$ of $V$ as the pair   
    
    \[
    \mathrm{sdim}_\K(V) = r \mid s, \quad \text{ where } \quad r:= \dim_\K(V\ev)\text{  and }s:= \dim_\K(V\od).
    \]
    
\end{definition}

\begin{proposition}
   Let $(R, \m)$ be a local  Noetherian superring. Then, $\mathrm{sdim}_{R/\m}(\m/\m^2)$ is finite. 
\end{proposition}

\begin{proof}
 Because $\m$ is a finitely generated $R\ev$-module, we can apply Nakayama's Lemma to conclude that $\m/\m^2$ is finitely generated as an $R/\m$-module.
\end{proof}

\begin{definition}
    Let $\K$ be a field and $R$ a superring. If $R$ is also a $\K$-supervector space with the same grading, we call $R$ a $\K$-\textit{superalgebra}.
\end{definition} 

\begin{example}\label{ex:ksuperalg}
Let $\K$ be a field and consider 

\[
R=\K[X_1, \ldots, X_s \mid  \theta_1, \ldots, \theta_d]:=\K\langle Z_1, \ldots, Z_s, Y_1, \ldots, Y_d\rangle/(Y_iY_j+Y_jY_i, Z_iZ_j-Z_jZ_i, Z_iY_j-Y_jZ_i).
\]

This $\K$-superalgebra is called the \textit{polynomial superalgebra with coefficients in} $\K$, with \textit{even indeterminants}\\ $X_1, \ldots, X_s$ and \textit{odd indeterminants} $\theta_1, \ldots, \theta_d$. Here, $\theta_i$ corresponds to the image of $Y_i$ and $X_j$ corresponds to the image of $Z_j$ in the quotient $\K$-algebra from the equality above, with $i=1, \ldots, d$ and $j=1, \ldots, s$.


The super-reduced of $R$ corresponds to $\overline{R}\cong\K[X_1, \ldots, X_s]$ and it is easy to check that any element of $R$ can be written in the form 


\[
f=f_{i_0}(X_1, \ldots, X_s)+\sum_{J\,:\,\text{even}}f_{i_1\cdots i_J}(X_1, \ldots, X_s)\theta_{i_1}\cdots\theta_{i_J}+\sum_{J\,:\,\text{odd}}f_{i_1\cdots i_J}(X_1, \ldots, X_s)\theta_{i_1}\cdots\theta_{i_J},
\]


\noindent where $f_{i_0}, f_{i_1\cdots i_J}\in\K[X_1, \ldots, X_s]$ for all $J$.  
\end{example}

A generalization of the previous example is given below. 

\begin{example}
     Let $R$ be a commutative ring and $M$ an $R$-module whose elements are labeled as odd. Consider the tensor superalgebra 

    \begin{align*}
    T_R(M)&:=\bigoplus_{n\geq0}M^{\otimes n}. 
    \end{align*}

    \noindent If $n\geq0$ consider the supersubmodule $N_n$  of $M^{\otimes n}$ generated by the elements 
    
    \[
    v_1\otimes v_2+(-1)^{|v_1||v_2|}v_2\otimes v_1,\quad\text{  where }\quad v_1, v_2\in h(M^{\otimes n}).
    \]
    
    Let 

    \[
    N:=\bigoplus_{n\geq0}N_n
    \quad\text{ and }\quad \bigwedge_R(M):=T_R(M)/N.
    \]

    \noindent If $M$ has a free basis $\theta_1, \ldots, \theta_d$, we define 

    \[
    R[\theta_1, \ldots, \theta_d]:=\bigwedge_R (M).
    \]

    The above construction corresponds to the \textit{polynomial $R$-superalgebra with odd indeterminants} $\theta_1, \ldots, \theta_d$. 
    
    Now, as a generalization of Example \ref{ex:ksuperalg}, given a $\K$-superalgebra $R$, we define  the so-called \textit{polynomial $R$-superalgebra in even indeterminants} $X_1, \ldots, X_s$ and \textit{odd indeterminants $\theta_1, \ldots, \theta_d$} as 


    \[
    R[X_1, \ldots, X_s\mid\theta_1, \ldots, \theta_d]:=R\otimes \K[X_1, \ldots, X_s][\theta_1, \ldots, \theta_d].
    \]
\end{example}

\begin{remark}\label{rem:splitring}
Polynomial superalgebras are the prototypical example of an important subclass of superrings, which are called \textit{split superrings}. More precisely, a superring $R$ is called {\it split} if the short exact sequence

\[
0\longrightarrow \J_R\longrightarrow R\longrightarrow \overline{R}\longrightarrow0
\]

\noindent is a split exact sequence or, equivalently, $R\cong\overline{R}\oplus\J_R$ (see \cite[p.7]{bruzzo2023notes} and  \cite[\S 3.5]{westrathesis}).  
\end{remark}


\subsection{Associated primes} 


\begin{definition}
Let $R$ be a superring and $M$ an $R$-supermodule. A prime ideal $\p$ of $R$ is \emph{associated to  $M$} if $\p$ is the annihilator of some nonzero homogeneous element $m$ of $M$, i.e.,  $\p=\ann_R(m):=\{x\in R\mid xm=0\}$.

Observe that $M$ admits the structure of both an $R$-supermodule and an $R\ev$-module. The associated primes of $M$ with respect to these structures will be denoted by $\ass_R(M)$ and $\ass_{R\ev}(M)$, respectively.
\end{definition}

\begin{remark}
    Let $R$ be a superring. Because any prime ideal $\p$ of $R$ is of the form $\p\ev\oplus R\od$, then there is bijection  between $\spec(R)$ and $\spec(R\ev)$. However, the property of being an associated prime is not preserved by this bijection. For example, let $R=A[\theta_1, \theta_2, \theta_3]/(\theta_1\theta_2, \theta_2\theta_3),$ where $A$ is a domain and $\theta_1, \theta_2$ and $\theta_3$ are odd variables. If  $\overline{\theta_i}$ denoted the class of $\theta_i$ in $R$, then  $R\ev=A[\,\overline{\theta}_1 \overline{\theta}_3]$ and $\ann_{R\ev}(\overline{\theta}_1)=(\overline{\theta}_1\overline{\theta}_3)_{R\ev}$ is an associated prime whereas $\ann_R(\overline{\theta}_1)=(\overline{\theta}_1, \overline{\theta}_2)_R$ is not prime in $R$.
\end{remark}

\begin{definition}
Let $R$ be a superring and $M$ an $R$-supermodule. An element $a\in R-\{0\}$ is called a \textit{zerodivisor on} $M$ if there exists some nonzero homogeneous element $m\in M$ such that $am=0$.
\medskip

We denote by $D(M)$ the set consisting of $0$ and the set of zerodivisors on $M$.
\end{definition}

As in the commutative case, if $R$ is a Noetherian superring, we have a natural relation between the set of zerodivisors on $M$ and associated primes:

\begin{proposition}\label{Cor:2.8} Let $R$ be a Noetherian superring and let $M$ be a nonzero $R$-supermodule. Then  

\[
\displaystyle D(M)=\displaystyle\bigcup_{\p\in\ass(M)}\p.
\]
\qed 
\end{proposition}

\begin{proposition}\label{Proposition:0.1}
    $\p\in\ass(M)$ if and only if there is an injective morphism $R/\p\hookrightarrow M$ or $R/\p\hookrightarrow \Pi M$.
\end{proposition}

\begin{proof} Suppose that $\p\in\ass(M)$ and let $m\in h(M)$ be nonzero such that $\p=\ann_R(m)$. Consider the homomorphism $\phi:R\to M$, $r\mapsto rm$. Note that $\p=\ker(\phi)$. Thus, we find a injective homomorphism $R/\p\hookrightarrow M$. If $m$ is even, this is parity preserving as a morphism of $R$-supermodules $R/\p\hookrightarrow M$. If $m$ is odd, this is parity preserving as a morphism of $R$-supermodules $R/\p\hookrightarrow\Pi M$. Conversely, if there is an injective morphism  $\phi:R/\p\hookrightarrow M$ or $\phi:R/\p\hookrightarrow \Pi M$. In any case, consider $m:=\phi(1)$. Then, $\p=\ann_R(m)$ and therefore $\p\in\ass(M)$.  
\end{proof}

\begin{proposition}[Prime avoidance]\label{Prop:PrimeAvoidance}
Let $\a, \a_1, \ldots, \a_d$ be superideals of a superring $R$, where at most two of the $\a_i$ are not prime. If $\a\subseteq\bigcup_{1\leq i\leq d} \a_i,$ then $\a$ is contained in one of the $\a_i$.    
\end{proposition}

\begin{proof}
    Recall that any prime ideal is completely prime (Remark \ref{rmk:prime-ideal} iv)), so the proof is same as in the commutative case (e.g.,  \cite[Lemma 3.3]{eisenbud}).
\end{proof}

\begin{corollary}\label{Cor3}
    Let $R$ be a Noetherian superring, $M$ a nonzero $R$-supermodule and $\a$ a superideal of $R$. Then $\a$ contains a non-zerodivisor on $M$ or it annihilates an element of $M$.
\end{corollary}

\begin{proof}
Let $\a$ be  a superideal consisting of zerodivisors on $M$. Then, by Proposition \ref{Cor:2.8}, $\a$ is contained in the union of the associated primes of $M$. Thus, by Proposition \ref{Prop:PrimeAvoidance}, $\a$ must be contained in $\ann(m)$ for some  $m\in M$.
\end{proof}

Let $R$ be a Noetherian superring, $\p$ a prime ideal of $R$ and consider $M$ a finitely generated $R$-supermodule. Let $S\subseteq R\ev$ be a multiplicative set. The  proof of the next proposition relies upon the following result, which will be proved (independently) in the next section: 

\begin{equation}\label{ASS}
    \text{$\p\in\ass(M)$ and $\p\cap S=\emptyset$ \quad if and only if \quad $S^{-1}\p\in\ass_{S^{-1}R}(S^{-1}M)$. }
\end{equation}

\begin{proposition}\label{prop1}  
\it Let $R$ be a Noetherian superring. Then $K(R)$ has finitely many maximal ideals, and such ideals are the localizations of the maximal associated primes. 
\end{proposition}

\begin{proof}
Let $S$ be such that $K(R)=S^{-1}R$. Consider $\mathfrak{M}$ any maximal ideal of $K(R)$. It is not hard to see that    any element $m/r\in\mathfrak{M}$ is a zerodivisor on $K(R)$. It then follows that the prime ideal $\mathfrak{M}$ is formed by zero divisors, that is to say, 

\[
\mathfrak{M}\subseteq D_{K(R
)}(K(R))=\bigcup_{\mathfrak{P}\in\ass(K(R))}\mathfrak{P}.
\]
By prime avoidance, for some $\mathfrak{P}\in\ass_{K(R
)}(K(R))$, $\mathfrak{M}\subseteq\mathfrak{P}$. Thus, $\mathfrak{M}\in\ass_{K(R
)}(K(R))$, because $\mathfrak{M}=\mathfrak{P}$, for $\mathfrak{M}$ being maximal. Let $\mathfrak{M}=S^{-1}\m$. Then, $\m\cap  S=\emptyset$ and $\m\in\ass(R)$. That is, $\mathfrak{M}$ is the localization of an associated prime, as claimed.
\end{proof}


\subsection{Noetherian and Finitely Presented supermodules} 


\begin{definition}
Let $R$ be a superring and $M$ an $R$-supermodule. We call $M$ \emph{Noetherian} if it satisfies one   of the following equivalent conditions (\cite[Proposition 3.3.1]{westrathesis}):

\begin{itemize}
\item[i)] $M$ satisfies the ascending chain condition on its $R$-subsupermodules.
\item[ii)] Each $R$-subsupermodule $N$ of $M$ is finitely generated.
\item[iii)] Any non-empty collection of $R$-subsupermodules of $M$ has a maximal element.
\end{itemize}
\end{definition}
 
\begin{definition}\label{Def:loc.free}
Let $R$ be a superring and $M$ a supermodule over $R$.
\begin{itemize}

\item[i)] We say that $M$ is a \textit{free $R$-supermodule} if there are disjoint sets $I$ and $J$ such that 

\[
M\simeq\displaystyle\left(\bigoplus_{i\in I}R\right)\oplus\left(\bigoplus_{j\in J}\Pi R\right).
\]

\noindent We say that $M$ has \textit{finite rank} $n$ if $n=|I\cup J|$ is finite. We also say that $M$ has \textit{rank} $|I|\mid |J|$.

\item[ii)] We call $M$ a \textit{flat $R$-supermodule} if the functor $M\otimes_R $- is left-exact. 

\item[iii)] $M$ is called \emph{finitely presented} if there exists a short exact sequence of $R$-supermodules 
        \[
        0\to K\to F \to M\to 0
        \]
such that  $F$ has finite rank and $K$, the kernel of the morphism $F\to M$, is finitely generated. 
    \end{itemize} 
\end{definition}

Observe that, if $R$ is a Noetherian superring, then $M$ is a  finitely presented $R$-supermodule if and only if $M$ is finitely generated (\cite[Proposition 3.3.8]{westrathesis}).

Recall that given $\varphi: R\rightarrow R'$, a morphism of superrings, we can endow $R'$ with an $R$-supermodule structure induced by $\varphi$. Similarly, if $M$ is an $R$-supermodule, then $M_{R'}:=R'\otimes_R M$ admits a natural structure of $R'$-supermodule, where on $R'$ we assume the $R$-supermodule structure induced by $\varphi$.

\begin{remark}\label{rem:devissage}
Let $R$ be a Noetherian superring and consider $M$ and $N$ finitely generated $R$-supermodules. Since $S^{-1}R$ is a flat $R$-supermodule and the localization functor is exact, we find from \cite[Theorem 6.5.9]{westrathesis}, the isomorphism  $
S^{-1}R\otimes_R \homm_R(M, N)\stackrel{\simeq}{\longrightarrow}{\homm_{S^{-1}R}(S^{-1}R \otimes_R M,\, S^{-1}R\otimes_R N)}.
$ In particular, if $\p$ is a prime ideal of $R$, we obtain  $
\homm_R(M, N)_\p\stackrel{\simeq}{\longrightarrow}{\homm_{R_\p}(M_\p,\, N_\p)}.
$  
\end{remark} 

\begin{proposition}\label{Prop2}
Let $R$ be a  Noetherian superring with finitely many maximal ideals $\m_1,\ldots,\m_d$, and let $M$ and $N$ be finitely generated 
$R$-supermodules. If $M_{\m_i} \simeq N_{\m_i}$, for each $i=1,\ldots, d$, then $M \simeq N$.  
\end{proposition}

\begin{proof} The proof utilizes the same logical steps as the commutative case (\cite[Solution to Exercise 4.13]{eisenbud}).
\end{proof}

To conclude this section, we proceed to prove the equivalence in \eqref{ASS}.   

\begin{proposition}\label{Proposition:0.2} Let $R$ be a Noetherian superring, $\p$ a prime ideal and $M$ a finitely generated $R$-supermodule. Then, $\p\in\ass(M)$ and $\p\cap S=\emptyset$ if and only if $S^{-1}\p\in\ass_{S^{-1}R}(S^{-1}M)$. 
\end{proposition}

\begin{proof}
   The left-to-right implication follows from Proposition~\ref{Proposition:0.1}, together with the condition $\p \cap S = \emptyset$.  
The converse direction also follows easily from Proposition~\ref{Proposition:0.1} and Remark~\ref{rem:devissage}.
\end{proof}

\subsection{Krull superdimension and regularity}\label{Subsec:Krull:SuperDim}  


In this subsection we introduce  the superdimension theory of superrings recently developed by Masuoka and Zubkov \cite{MASUOKA2020106245} (see also  \cite{ZubkovKolesnikov}).

We denote by $\kdim(-)$ the usual Krull dimension in the commutative context. If $R$ is a commutative ring and $M$ is an $R$-module, we denote by $\text{dim}(M)$ \emph{the dimension of $M$ as $R$-module}.  

\begin{definition}
Let $R$ be a Noetherian superring such that $\kdim(R\ev)<\infty$ and let $a_1, \ldots, a_s\in R\od$. For any $I\subseteq\{1, 2, \ldots, s\}$, with $|I|=n$, let  $a^I:=a_{i_1}\cdots a_{i_{n}}$, where $I=\{i_1, \ldots, i_{n}\}$ and the product is taken in $R$. Let $\ann_{R\ev}(a^I)$ be the ideal of $R\ev$ consisting of all elements in $R\ev$ that annihilate $a^I$. 

\begin{itemize}
    \item[i)] If $\kdim(R\ev)=\kdim(R\ev/\ann_{R\ev}(a^I))$, we say that $a_{i_1}, \ldots, a_{i_{n}}$ form a \emph{system of odd parameters of length $n$}.
    
    \item[ii)] The greatest integer $m$ such that there exists a set of odd parameters of length $m$ is called the \emph{odd Krull superdimension} of $R$ and is  denoted by $m:=\ksdim\od(R)$. 
    
    \item[iii)] The \emph{even Krull superdimension} of $R$ is the usual Krull dimension of $R\ev$ and is denoted $\ksdim\ev(R)$.
    
    \item[iv)] The \textit{Krull superdimension} of $R$ is the pair  $\ksdim\ev(R)\mid\ksdim\od(R)$. 
\end{itemize}
\end{definition}

\begin{definition}
Let $(R, \m)$ be a local Noetherian superring, with $\K=R/\m$ its residue field. The superring $R$ is called 
\textit{regular} if 
$$ 
\ksdim(R)=\sdim_\K(\m/\m^2).
$$ 
\end{definition}

\begin{definition}\label{def:regular:superring}
Let $R$ be a Noetherian superring. We say that $R$ is \textit{regular} if it satisfies one of the 
following equivalent conditions (see  \cite[Corollary 5.4 and \S 5.2]{MASUOKA2020106245}).

    \begin{itemize}
        \item[i)] For every prime ideal $\p$ in $R$, the local superring $R_\p$ is regular.
        
        \item[ii)] For every maximal ideal $\m$ in $R$, the local superring $R_\m$ is regular.
    \end{itemize}   
\end{definition}

\begin{remark}
    Before finishing this section, it is worth mentioning that in \cite[Chapter 7, \S 7.1]{westrathesis}, Westra proposed another notion of dimension for a superring. Namely, if $(R, \mathfrak{m})$ is a Noetherian local superring, he  defined the following numbers:

\begin{itemize}
    \item  The  \textit{total dimension of} $R$, $\mathrm{Tdim}(R)$, which is the minimal number of generators of $\mathfrak{m}$.
    \item  The  \textit{bare dimension of} $R$, $\mathrm{Bdim}(R)$, which is the minimal number of generators of $\overline{\mathfrak{m}}$ in $\overline{R}$.
    \item The  \textit{Krull dimension of} $R$, $\kdim(R)$, which is given by the Krull dimension of $\overline{R}$. 
    \item The  \textit{odd dimension of} $R$, $\mathrm{Odim}(R)$, which is the minimal number of generators of $\mathfrak{J}_R$.
\end{itemize}

Westra proved that $\mathrm{Tdim}(R)=\dim_{R/\mathfrak{m}}(\mathfrak{m}/\mathfrak{m}^2)$. On the other hand, for the Krull superdimension of $R$ (as defined in \cite{MASUOKA2020106245}), we find that  $\ksdim(R)\leq\dim_{R/\mathfrak{m}}(\mathfrak{m}/\mathfrak{m}^2)$ and, $R$ is said to be \textit{regular} if the equality holds. 

We build upon the dimension theory presented in \cite{MASUOKA2020106245}.  Note, however, that if $R$ is a regular superring, the two definitions coincide. This fact is key in \S \ref{Sec:Dedekind}.   
\end{remark}


\section{Invertibility}\label{Sec:Invertibibility}

 
This section explores the concepts of invertible supermodules and fractional superideals, dividing itself into two subsections. In the first, we study the fundamental properties of these objects. In the second one, as an application of the theory developed in the first subsection, we construct the Picard  and Cartier groups  of a superring.  

Let $R$ be a commutative ring. Recall that if $S$ is the set of non-zerodivisors of $R$, then the localization $S^{-1}R$ is denoted by $K(R)$ and called the \textit{total ring of fractions of} $R$. If $R$ is furthermore a domain, $K(R)=S^{-1}R$, with $S=R-\{0\}$, is a field, known as the \textit{total field of fractions of} $R$, and there exists a natural embedding of $R$ into $K(R)$ given by the localization morphism. A natural question arises: given a superdomain $R$, is $K(R)$ a superfield? The answer is negative, as can be seen by considering  $R=\Bbbk[\, X\mid \theta_1, \,\theta_2\,]/(X\theta_1\theta_2)$. Furthermore, in general, it is not true that $K(R)=R_{\J_R}$ and we cannot consider $R$ as a subring of $R_{\J_R}$, because the localization morphism $R\to R_{\J_R}$ is not always injective. Therefore, to establish a robust theory of invertible supermodules and fractional superideals, we will work within $K(R)$ rather than $Q(R):=R_{\J_R}$. 

Recall that a superring $R$ is called a superdomain if $\overline{R}$ is a domain or, equivalently, $\J_R$ is a prime ideal. We say that $R$ is a \textit{strong superdomain}, if, additionally, $R$ satisfies one (and hence all) of the following equivalent conditions:

        \begin{enumerate}
        \item[i)] The localization map $R\hookrightarrow R_{\J_R}$ is injective. 
        
        \item[ii)]  Any element $r\in R-\J_R$ is a non-zerodivisor. 
        
        \item[iii)] Any element  $r\in R\ev-R\od^2$ is a non-zerodivisor.
    \end{enumerate}


\subsection{Invertible supermodules and Fractional superideals}


\begin{definition}
Let $R$ be a superring and $M$ an $R$-supermodule. We say that $M$ is \textit{invertible} if:

\begin{enumerate}
\item[i)] $M$ is finitely generated, and 
\item[ii)] for every prime ideal $\p$ of $R$, $M_\p\simeq R_\p $.
\end{enumerate}

\noindent In other words,  $M$ is invertible if it is finitely generated and locally free of rank $1\mid 0$.
\end{definition}

\begin{definition}
Let $R$ be a superring and let $K(R)$ be the total superring of fractions of $R$. 
 
\begin{enumerate}  
\item[i)] An $R$-subsupermodule $M$ of $K(R)$ is called a \textit{fractional superideal of} $R$ if there exists a homogeneous non-zerodivisor $d\in R$ such that $d M\subseteq R$.

\item[i)] If $M$ is an $R$-subsupermodule of $K(R)$, we denote by $M^{-1}$ the $R$-supermodule generated by the homogeneous elements $d\in K(R)$ such that $dM\subseteq R$.
\end{enumerate}
\end{definition}

Note that any superideal of $R$ is also a fractional superideal, just take $d=1$. In the purely even case, these are usually called \textit{integral ideals}.  For fractional superideals $M$ and $N$ of $R$, we denote by $(M:N)$, the $R$-supermodule generated by the homogeneous elements $z\in K(R)$ such that $zN\subseteq M$. Here $R^{\ast}$ stands for the set of units of the ring $R$. 

An element that is a non-zerodivisor on a ring is called \textit{regular}.\index{regular!element} A \textit{regular ideal}\index{regular!ideal} is an ideal containing a regular element. The following proposition is proved similar to \cite[Proposition 25.3]{altman2013term}.  Recall that $D(R)$ denotes the zerodivisors of $R$.

\begin{proposition}\label{P:3.2:ALTMAN}
 Consider the following conditions for an $R$-subsupermodule $M$ of $K(R)$.

    \begin{enumerate} 
    \item[\rm i)] There exist superideals $\a$ and $\b$ of $R$, where $\b$ is regular,  such that $M=(\a:\b)$.

    \item[\rm ii)] There exists a homogeneous $z\in K(R)^*$, such that $z M\subseteq R$.

    \item[\rm iii)] There exists a homogeneous $z\in R-D(R)$, such that $zM\subseteq R$.

    \item[\rm iv)] $M$ is finitely generated.  
    \end{enumerate}

Then, i), ii) and iii) are equivalent and iv) implies  i), ii) and iii). Moreover, if $R$ is  Noetherian, then i)-iv) are equivalent, and the equivalence between  i)-iv) implies that $R$ is Noetherian.  \qed 
\end{proposition}

\begin{remark}  Let $R$ be a superring and $M$ an $R$-subsupermodule of $K(R)$. 

\begin{itemize} 
    \item[i)] If $M$ is finitely generated, then $M$ is isomorphic to  a superideal of $R$, by Proposition \ref{P:3.2:ALTMAN}. Hence, we can view $M$ as a fractional superideal of $R$.

    \item[ii)] If $M$ is nonzero, then $M^{-1}$ is a fractional superideal of $R$. Indeed, since $M\neq0$, there exists a nonzero $y\in h(R)$ and $x\in R\ev-D(R\ev)$, such that $y/x\in M$. Then, $y=x\cdot y/x\in M\cap R$. Since $M^{-1}M\subseteq R$,  we find that $yM^{-1}\subseteq R$. In other words, $M^{-1}$ is a fractional superideal of $R$. 
\end{itemize}    
\end{remark}

Let $R$ be a Noetherian superring and $M$ an $R$-supermodule. The $R$-supermodule 
 $M^{\vee}=\homm_R(M, R)$ is called \textit{the dual supermodule of $M$}. In this case, we can construct a natural ``evaluation'' morphism  
\begin{eqnarray*}
    \mu:M^{\vee}\otimes_R M&\to&R\\
    \varphi\otimes a&\mapsto&\varphi(a).
\end{eqnarray*} 

Below we will show that if $M$ is an invertible $R$-supermodule, then $\mu$ is an isomorphism. The converse statement include one more case than in the purely even case. Namely, we will show that if $\mu$ is an isomorphism, then $M$ is finitely generated and locally free of rank $1\mid 0$ or $0\mid 1$.

Much of the reasoning involved below is inspired in the proof of \cite[Theorem 11.6]{eisenbud} and in the work on fractional ideals in \cite[Chapter 25]{altman2013term}.

\begin{proposition}\label{prop:Eisenbud11.6-a} Let $R$ be a Noetherian superring and $M$ any $R$-supermodule. If $M$ is invertible, then $\mu$  is an isomorphism. Conversely,  if $\mu$ is an isomorphism, then $M$ is finitely generated and locally free of rank $1\mid 0$ or $0\mid 1$. Moreover, if $\mu$ is an isomorphism and $M\otimes_R N\simeq R$ for some  $R$-supermodule $R$, then $N\simeq\homm_R(M, R)$. 
\end{proposition}

\begin{proof}

The left-to-right implication follows identical to the purely even case.  Conversely, suppose that $\mu$ is an isomorphism. We first show that $M$ is finitely generated. Indeed, since the morphism $\mu$ is surjective, for some nonzero homogeneous elements $\varphi_i\in M^\vee$ and $m_i\in M$ with $i\in\{1, 2, \ldots, n\}$, we have $1=\mu(\sum \varphi_i\otimes m_i)$ with $|\varphi_i|=|m_i|$ for all $i\in\{1, 2, \ldots, n\}$. We claim that $M$ is generated by the $m_i$. Consider  an arbitrary element $m\in M$. Then, $a_i:=\mu(\varphi_i\otimes m)\in R$. Let $\alpha$ be the following composition of isomorphism  

 
\[
M\simeq M\otimes_R R\xlongrightarrow{\sim} M\otimes_R(M^\vee\otimes_R M)\simeq(M\otimes_R M^\vee)\otimes_R M\simeq (M^\vee \otimes_R M)\otimes_R M\xlongrightarrow{\sim} R\otimes_R M\simeq M.
\]


Note that $\alpha:M\to M,\, m\mapsto\sum a_im_i$, is an isomorphism. In other words, the $m_i$ are homogeneous elements generating $M$.  Thus, $M$ is a finitely generated  $R$-supermodule.

It only remains to show that if $\p$ is a prime ideal of $R$, then $M_\p\simeq R_\p$ or $M_\p\simeq\Pi R_\p$. For this, we can restrict to the case when $R$ is local  (otherwise it suffices to use the fact that $\mu_\p$ is an isomorphism for all $\p$). Since $R$ is local, then $R-R^*$ is an ideal. This implies that for some $j\in\{1, 2, \ldots, n\}$, the element $u_j:=\mu(\varphi_j\otimes m_j)$ belongs to $R^*$. Let $\phi_j=u^{-1}_j\varphi_j$ and $m=m_j$. Note that $\phi_j$ is homogeneous with the same grading as $\varphi_j$. Further, $\mu(\varphi_j\otimes m)=1$ and the morphism  $\gamma:M\to R$ defined by $x\mapsto \mu(\varphi\otimes x)$ sends $m\mapsto 1$. Hence, $\gamma$ is surjective. Moreover, the morphism $\sigma:R\to M, a\mapsto am$, is such that $\sigma\circ\gamma=\alpha$. Since $\alpha$ is injective, we conclude that $\gamma$ is  injective, \textit{a fortiori}, an  isomorphism. If $\gamma$ is even, then $M\simeq R$, and if $\gamma$ is odd, then $M\simeq\Pi R$. 
\end{proof}

\begin{proposition}\label{P:Eisenbud:11.6-b}  Let $R$ be a Noetherian superring.

\begin{enumerate} 
    \item[\rm i)] Any invertible $R$-supermodule is isomorphic to a fractional superideal of $R$.
    \item[\rm ii)] Any invertible fractional superideal contains an even non-zerodivisor of $R$.
\end{enumerate}     
\end{proposition}

\begin{proof} i) Suppose that $M$ is an invertible $R$-supermodule. Since $R$ is Noetherian, by Proposition \ref{prop1}, $K(R)$ has a finite number of maximal ideals, that are the localizations of some maximal associated primes of $R$. Let $S$ be the set of non-zerodivisors of $R\ev$. Hence, if $\p$ is an associated prime of $R$, then all the elements of $\p$ are zerodivisors, {\it a fortiori}, $\p\cap S=\emptyset$. We find that $S^{-1}R=K(R)$, $S^{-1}\p=\p K(R)$ and $K(R)_{\p K(R)} \simeq R_\p$. By \cite[Proposition 5.1.16]{westrathesis}, if $M$ is an  $R$-supermodule which is invertible, then 

\[
M\otimes_R K(R)_{\p K(R)}\simeq M_\p\simeq R_\p\simeq K(R)_{\p K(R)}.
\]

\noindent Thus,  we obtain the isomorphisms 

\[
(M\otimes_R K(R))_{\p K(R)} \simeq M\otimes_R K(R)_{\p K(R)} \simeq K(R)_{\p K(R)},
\]

\noindent 
which implies that $M\otimes_R K(R) \simeq K(R)$, by Proposition \ref{Prop2}. Now, we claim that the localization map $$\lambda_M^S:M\to K(R)\otimes_R M \simeq S^{-1}M$$ is injective. Indeed, for each prime ideal $\p$ of $R$, the morphism $(\lambda_M^S)_\p:M_\p \simeq R_\p\to K(R)\otimes_R R_\p \simeq S^{-1}R_\p$ is injective, and consequently, we find that $\lambda_M^S$ is injective. Thus, we use the embedding $M\hookrightarrow K(R)\otimes_R M$ and the isomorphism $K(R)\otimes_R M \simeq K(R)$, to identify $M$ isomorphically with an $R$-subsupermodule $N$ of $K(R)$. In particular, $N$ is finitely generated and, since $R$ is Noetherian (by Proposition \ref{P:3.2:ALTMAN} iii)) there exists $d\in S$ such that $dN\subseteq R$, i.e., $N$ is a fractional superideal of $R$.

ii)  We will show the contrapositive. Suppose that $M\subseteq K(R)$ is a fractional superideal such that $M\cap R$ consists of zerodividors. There is a homogeneous non-zerodivisor $u\in R$ such that $uM\subseteq R$. By Corollary \ref{Cor3}, there exists a nonzero homogeneous  element $b\in R$ annihilated by $uM$. Hence, it follows that $M$ is annihilated by $ub$. Now, by Proposition \ref{Cor:2.8}, there exists a prime $\p$, which contains the annihilator of $ub$. Suppose, for a contradiction, that there is an isomorphism $\phi: R_\p\to M_\p$ and consider any $r\in R-\p$. Then, $\phi(ubr)=ub\phi(r)=0$, since $ub$ annihilates $M$. Since $\phi$ is an isomorphism, we find that $ubr=0$, that is, $r$ annihilates $ub$, contradicting the choice of $\p$ and of $r$.  Therefore,  $M_\p\not \simeq R_\p$,
and it follows that $M$ is not invertible. Therefore, $M$ contains a non-zerodivisor. Since $M$ is $\Z_2$-graded, we can assume that it contains an even non-zerodivisor.  
\end{proof}
 
In the following, unless otherwise stated, $R$ denotes a Noetherian superring. 

\begin{proposition}\label{P:Eisenbud:11.6-c} Let $R$ be a Noetherian superring. If $M, N$ are  $R$-subsupermodules of $K(R)$ which are invertible, then the natural morphism

\begin{eqnarray*}
    \phi:M^{-1}N&\to&\homm_R(M, N)\\
    t&\mapsto&\function{\phi_t}{M}{N}{m}{tm}
\end{eqnarray*} 

\noindent is an isomorphism of $R$-supermodules.
\end{proposition} 

\begin{proof}     By Proposition \ref{P:Eisenbud:11.6-b} ii), there exists $v\in M$ such that $v$ is an even non-zerodivisor of $R$. Then, for any nonzero homogeneous element $t\in M^{-1}N$, we have $vt\neq 0$. Hence, $t$ induces a nonzero element in $\homm_R(M, N)$. Hence, the morphism $\phi:M^{-1}N\to\homm_R(M, N)$ is injective. It remains to show that $\phi$ is surjective. For this, consider any $\varphi\in\homm_R(M, N)$ and let $w\in N$ be such that $\varphi(v)=w$. We claim that $\varphi=\phi_{w/v}$, that is, $\varphi$ is the image of $w/v$ under $\phi$. It is clear that $\varphi(v)=\phi_{w/v}(v)$, so the proof of the proposition follows from the following more general claim: If two morphisms $\alpha$ and $\beta$ from $M$ to $K(R)$, are such that $\alpha(v)=\beta(v)$, then $\alpha=\beta$. Indeed, it suffices to show that $\alpha_\p=\beta_\p$ for every prime ideal $\p$ of $R$. Note that since $v$ is a non-zerodivisor, the annihilator of $v$ must be zero, so the image of $v$ in $M_\p\simeq R_\p$ is a non-zerodivisor. In this case, $\alpha_\p(v)=\beta_\p(v)$. Thus, we find that $v\alpha_\p(1)=\alpha_\p(v)=\beta_\p(v)=v\beta_\p(1)$, that is to say, $\alpha_\p(1)=\beta_\p(1)$, so  $\alpha_\p=\beta_\p$. This completes the proof of the claim. The proposition follows from the claim. 
\end{proof}

\begin{corollary}\label{Corollary:3.9}  If $M$ is an invertible $R$-subsupermodule of $K(R)$,  then $M^\vee\simeq M^{-1}$.
\end{corollary}

\begin{proof}
    Note that $M^{-1}R=M^{-1}$, because $M^{-1}$ is an $R$-supermodule. Since $M$ is invertible, $M^\vee\otimes_R M\simeq R$, by Proposition \ref{prop:Eisenbud11.6-a}, and we find that $M^\vee \simeq M^{-1}R=M^{-1}$, by Proposition \ref{P:Eisenbud:11.6-c} with $N=R$.
\end{proof}

\begin{proposition}\label{P:Eisenbud:11.6-d} Let  $M$  be an $R$-subsupermodule of  $K(R)$. Then, $M^{-1}M=R$ if and only if $M$ is invertible. 
\end{proposition}

\begin{proof}  The ``only if'' part: Suppose that $M\subseteq K(R)$ is an $R$-supermodule with $M^{-1}M=R$. First, let us show that the conclusion holds when $R$ is local with maximal ideal $\m$. Since $M^{-1}M=R$, we can find finitely many $a_i$ in $h(M^{-1})$ and $x_i$ in $h(M)$ such that $1=a_1x_1+\cdots+a_rx_r$. If for all $i$, $a_ix_i\in\m$, then $1\in\m$, which is impossible. Then, there exists a homogeneous element $a_i\in M^{-1}$ such that $a_iM\not\subseteq\m$. Consequently, $a_iM=R$, which implies that $a_i$ is a non-zerodivisor of $R$. We may assume that $a_i$ is even. Then, multiplication by $a_i$ defines an isomorphism between $M$ and $R$, which proves the claim. In the general case, if $R$ is not local, we apply the previous reasoning to $R_{\p}$, with $\p$ running through the set of prime ideals of $R$.

The ``if'' part: Note that $M^{-1}M$ is a superideal of $R$. We must show that $1\in M^{-1}M$. For this, consider $\varphi_i\in M^\vee$ and $a_i\in M$ ($i=1, \ldots, n$) such that $1=\varphi_1(a_1)+\cdots+\varphi_n(a_n)$. If $\psi_i$ is the element of $M^{-1}$ corresponding to $\varphi_i$, through the isomorphism $M^{-1}\simeq M^\vee$, then $1=\psi_1a_1+\cdots+\psi_na_n$. Thus, $M^{-1}M=R$, as claimed. 
\end{proof}

\begin{proposition}\label{LEMA:3.11}
    Let $R$ be a local Noetherian superring and $M$ a fractional superideal of $R$. Then, $M$ is invertible if and only if, $M$ is principal generated by some unit $x\in K(R)^{*}$.
\end{proposition}

\begin{proof}       
   The ``only if'' part:  Since $R$ is local, $R-R^*$ is an ideal of $R$. Assuming that $M$ is invertible, there exist elements $a_i\in h(M^{-1})$ and $x_i\in h(M)$   such that $1=a_1x_1+\cdots+a_sx_s$. Thus, for some $j\in\{1, 2, \ldots, s\}$, we find that $a_jx_j\not\in R-R^*$. Note that every $x\in M$ can be written as $x=[(xa_j)(a_jx_j)^{-1}]x_j$, that is, $M=x_jR$ with $x_j\in K(R)^*$. 

    The ``if'' part: If $M=xR$ with $x\in K(R)^{*}$, then $M^{-1}=\frac{1}{x}R$ and $M^{-1}M=R$. Thus, $M$ is invertible by Proposition \ref{P:Eisenbud:11.6-d}, so the proposition follows.
\end{proof}

\begin{proposition}\label{Inv:is:local} Let $M$ be a fractional superideal of $R$. The following statements are equivalent.

\begin{itemize}
    \item[i)] $M$ is an  invertible  $R$-supermodule. 
    \item[ii)] $M$ is finitely generated and, for any prime ideal $\p$ of $R$, $M_\p$ is an  invertible  $R_{\p}$-supermodule.
    \item[iii)] $M$ is finitely generated and, for any maximal ideal $\m$ of $R$, $M_\m$ is an invertible  $R_\m$-supermodule.
\end{itemize} 
\end{proposition}

\begin{proof}  i) $\Rightarrow$ ii): It is easy to see that for any multiplicative subset $S$ of $R\ev$, we have $S^{-1}(R:M)=(S^{-1}R:S^{-1}M)$ and further, $M^{-1}=(R:M)$ (cf. \cite[Lemma 25.4 and Proposition 25.8]{altman2013term}). Now, let $\p$ be any prime ideal of $R$. Then,  $R_\p=(M\cdot (R:M))_\p\simeq  M_\p(R:M)_\p \simeq M_\p(R_\p:M_\p)\simeq M_\p M_\p^{-1},$ that is, $M_\p$ is invertible.

ii) $\Rightarrow$ iii): Obvious. 

iii) $\Rightarrow$ i): Suppose that $M_\m$ is invertible for all maximal ideal $\m$ of $R$ and consider $N:=MM^{-1}$. Notice that $N$ is a superideal of $R$ such that $N_\m=M_\m M^{-1}_\m=R_\m$, for any maximal ideal $\m$. In particular, we have $N\not\subseteq\m$ and, consequently, $N=R$. Thus, by Proposition \ref{P:Eisenbud:11.6-d}, $M$ is an invertible $R$-supermodule. 
\end{proof}

\begin{proposition}
     Let $R$ be a Noetherian superring and $M$ a fractional superideal of $R$. Then, $M$ is invertible if and only if $M$ is finitely generated and locally of the form $xR_\p$, for some $x\in K(R_\p)^*$.
\end{proposition}

\begin{proof}
   The ``only if'' part: Suppose that $M$ is invertible, then $M^{-1}M=R$. Thus, $M_\p$ is invertible for every prime ideal $\p$. Since $R_\p$ is a local superdomain, $M_\p$ has the form $xR_\p$, for some $x$ in $K(R_\p)^*$. 

    The ``if'' part: Suppose that for all prime $\p$ of $R$, $M_\p=xR_\p$, with $x\in K(R_\p)^*$. Then, $M_\p^{-1}M_\p=R_\p$. That is, $M_\p$ is invertible for all $\p$. Since invertibility is local and $M$ is finitely generated, $M$ is invertible. 
\end{proof}

Observe that the conclusion of the previous proposition remains valid for maximal ideals.  
\medskip

Having covered the necessary invertible supermodules properties, we will now consider the Picard group of a superring in the next subsection. 

\subsection{Picard group and Cartier divisors}\label{Sub:sec:3.2}  


Let $R$ be a Noetherian superring and consider $M$ and $N$ invertible $R$-supermodules. Denote by $[M]$ the isomorphism class of $M$. We define the \textit{even Picard supergroup $\spic_+(R)$ of $R$} to be the group of isomorphisms classes of invertible $R$-supermodules under tensor product, that is,  $[M]\cdot[N]:=[M\otimes_R N]$. Then, the binary operation  ``$\cdot$'' endows $\spic_+(R)$ with an Abelian group structure,  where the identity is the class of $R$ and the inverse of the class of $M$ is the class of $M^{\vee}$.  Similarly, the set $\sC(R)$ of invertible $R$-subsupermodules of $K(R)$ can be endow with a group structure, in which the product and inverse are defined as in Remark \ref{rem:prodK} and Proposition  \ref{P:Eisenbud:11.6-d}, respectively. The group $\sC(R)$ is called the \textit{Cartier group of} $R$. 

The following proposition will be useful in order to describe the principal Cartier divisors on a Noetherian superring. 

\begin{proposition}\label{PROPO:3.13}
    Let $R$ be a Noetherian superring and consider $M$ and $N$ to be invertible $R$-subsupermodules of $K(R)$ such that $M\simeq N$. Then, there is some $u\in K(R)\ev^*$ such that $M=uN$. 
\end{proposition}

\begin{proof}  Let $\phi:M\to N$ be an isomorphism. By Proposition \ref{P:Eisenbud:11.6-c}, $\homm_R(M, N)\simeq M^{-1}N$, so we may identify the isomorphism $\phi$ with the multiplication for some element $u\in K(R)$. Further, since $\phi$ preserves the parity and from the equality $\homm_R(M, N)\ev=\hom_R(M, N)$, we find that $u$ is even. A similar reasoning on $\phi^{-1}:N\to M$ gives us a $v\in K(R)$ such that $uv=1$, because $\phi\circ\phi^{-1}$ is the identity map on $N$.  Then, $u$ is a unit and $M=uN$, as claimed. 
\end{proof}

Let $R$ be a Noetherian superring and consider the group homomorphism  $\psi:\sC(R)\to\spic_+(R)$ that assigns to an invertible $R$-subsupermodule of $K(R)$  its isomorphism class. Note that, by Proposition \ref{P:Eisenbud:11.6-b}, $\psi$ is surjective. Let us  describe the subgroup $\ker(\psi)$. First, note that if $u\in K(R)\ev^*$, then $uR \simeq R$; that is,  $uR$ is invertible and further  $uR\in\ker(\psi)$. Conversely, let $M\in\ker(\psi)$. Then, $M\simeq R$ and by Proposition \ref{PROPO:3.13},  $M=uR$ for some $u\in K(R)\ev^*$. Therefore, $\ker(\psi)=\{\,uR\mid u\in K(R)^*\ev\,\}.$ Every element $M\in \ker(\psi)$ is called a \textit{principal Cartier divisor on} $R$ and the group $\ker(\psi)$ is called the \textit{group of principal Cartier divisors on} $R$.  Observe that, as in the commutative case (see e.g., \cite[\S25.22]{altman2013term} or \cite[Corollary 11.7]{eisenbud}), we have the isomorphism   $\ker(\psi)\simeq K(R)\ev^*.$ The quotient group  $\sC(R)/K(R)\ev^*$ is called the \emph{ideal class group of $R$}. Thus, the ideal class group is canonically isomorphic to the even Picard supergroup: $\spic_+(R)\simeq\sC(R)/K(R)\ev^*.$
 
\begin{remark}
    It is worth noting that, stemming from the reasoning in the proof of Proposition \ref{P:Eisenbud:11.6-d}, if $R$ is a local superring, then any  invertible fractional superideal of $R$ is isomorphic to $R$. Consequently, in this case, both $\sC(R)$ and $\spic_+(R)$ are  trivial.
\end{remark}

Let $R$ and $R'$ be Noetherian superrings, $\varphi:R\longrightarrow R'$ a morphism and $M$ an invertible $R$-supermodule. By extension of scalars, $M\otimes_R R'$ admits an $R'$-supermodule structure, such that it is a  finitely generated $R'$-supermodule. Now, let $\q$ be a  prime ideal of $R'$. We obtain $(M\otimes_R R')_{\q}\simeq R'_\q$, i.e., $M\otimes_R R'$ is an invertible $R'$-supermodule. In particular, the map $\varphi_{\ast}:\spic_+(R)\to\spic_+(R')$ which associates the class of $M$ to the class of $M\otimes_R R'$ is a well-defined  group homomorphism.   

If we denote by $\srings$ the category of Noetherian superrings and by $\mathsf{aGrps}$ the category of Abelian groups, then we have a covariant functor 

\begin{eqnarray*}
\spic_+:\srings & \to & {\mathsf{aGrps}} \\ \nonumber
       R &\mapsto & \spic_+(R) \\ \nonumber
       (\varphi:R\to R') & \mapsto & (\varphi_{\ast}:\spic_+(R)\to\spic_+(R')) \nonumber
\end{eqnarray*}

\noindent Indeed, for any Noetherian superring $R$, $(\id_{R})_\ast=\id_{\spic_+(R)}$. Moreover, if $R, R', R''$ are Noetherian superrings and  $\varphi:R\to R'$ and $\psi:R'\to R''$ are morphisms, then $(\psi\circ\varphi)_{\ast}=\psi_{\ast}\circ\varphi_{\ast}$. 

It easily follows the following proposition. 

\begin{proposition} Let $R$ be a Noetherian split superring. Then, the projection $\pi_R:R\to\overline{R}$ induces a surjective group homomorphism $\pi_{R\ast}:\spic_+(R)\to\spic_+(\overline{R})$ and $\spic_+(R)$ splits as \[
\spic_+(R) \simeq\spic_+(\overline{R})\oplus\ker((\pi_{R})_\ast).
\]
\end{proposition}

Note that if $R\od=0$, then $\spic_+(R)$ is simply the usual Picard group of $R$.

\begin{example}
     Let $R$ be an integral domain and $R'=R[X_1, \ldots, X_n\mid \theta_1, \ldots, \theta_m]$. Then, 


    \[
    \spic_+(R')\simeq\spic_+(R[X_1, \ldots, X_n])\oplus\ker((\pi_{R'})_*).
    \]


    It is known in commutative algebra that 

    
    \[
    \spic_+(R[X_1, \ldots, X_n])\simeq\spic_+(R)\oplus\bigoplus_{1\leq i\leq n+1}\mathbf{NPic}_i(R),
    \]

    
    \noindent where $\mathbf{NPic}_i(R)$ is the kernel of the map $\spic_+(R[X_1, \ldots, X_{i+1}])\to\spic_+(R[X_1, \ldots, X_i])$, induced by the evaluation map at $x_{i+1}=0$ (see e.g., \cite[\S2]{coyjim}). Thus, 


    \[
    \spic_+(R')\simeq\spic_+(R)\oplus\left(\bigoplus_{1\leq i\leq n+1}\mathbf{NPic}_i(R)\right)\oplus\ker((\pi_{R'})_*).
    \]
 
\end{example}

\begin{example}
    Let $R$ be a commutative ring and consider $R'=R[x, x^{-1}]$. It was proved by C. A. Weibel in  \cite{weibelPIC}, that $\spic_+(R')$ admits the decomposition 
    
    \[
    \spic_+(R')\simeq\spic_+(R)\oplus\mathbf{NPic}(R)\oplus\mathbf{NPic}(R)\oplus H^{1}(R),
    \]

\noindent where $H^1(R)$ is the étale cohomology group $H^1_{\text{ét}}(\spec(R), \Z)$ and $\mathbf{NPic}(R)$ is the kernel of $\spic_+(R[x])\to\spic_+(R)$ induced by the evaluation map $R[x]\to R, x\mapsto1$. Thus, for example, 

\begin{align*}
    \spic_+(R[x, x^{-1}][\theta_1, \ldots, \theta_n])&\simeq\spic_+(R[x, x^{-1}])\oplus\ker(\pi_*)\\
    &\simeq\spic_+(R)\oplus\mathbf{NPic}(R)\oplus\mathbf{NPic}(R)\oplus H^{1}(R)\oplus\ker(\pi_*),
\end{align*}

\noindent where $\pi_*$ is induced by the canonical map $R[x, x^{-1}][\theta_1, \ldots, \theta_n]\to R[x, x^{-1}]$.   
\end{example}
  

\section{Projective supermodules}\label{Sec:Projective:Supermodules}


Projective supermodules take center stage in this section, where we define them, cover their properties, and further demonstrate their connection to invertible supermodules. Given the in-depth treatment of this topic in \cite{morye2022notes} and \cite[Chapter 6]{westrathesis}, we restrict our focus here to the aspects crucial for understanding the remaining of this paper. 

\begin{definition}
An $R$-supermodule $M$ is \textit{projective} if the functor $\homm_R(M, \cdot):\mathrm{sMod}_R\to\mathrm{sMod}_R$ is exact.
\end{definition}

The following proposition follows similar to \cite[Proposition 1.35]{narkiewicz2004elementary}. 

\begin{proposition}[Dual basis theorem]\label{DualBasisTheorem} Let $R$ be a superring and $M$ an $R$-supermodule. Then $M$ is projective  if and only if   there exist collections $
\{a_i\mid i\in I\}\subseteq h(M)$ and $\{\phi_i\mid i\in I\}\subseteq h(\homm_R(M, R))$ such that every $m\in M$ has the form  

\[
m=\sum_{i\in I}\phi_{i}(m)a_i,
\]

\noindent with all but finitely many $\phi_i(m)$ are zero. 
\end{proposition}

As in the commutative case, Proposition \ref{DualBasisTheorem} has an interesting application. Namely, it builds a connection between the notions of invertibility and projectivity on fractional superideals, as we find in the following two propositions, whose proofs follow arguments very close to their commutative counterpart. 

\begin{proposition}
Let $R$ be a superdomain and $M$ an invertible fractional superideal of $R$. Then, $M$ is projective. In particular, if $R$ is Noetherian, then any invertible $R$-supermodule is projective. 
\end{proposition}

\begin{proof} 
Let $M\neq\J_R$ be a nonzero fractional superideal of $R$. If $M$ is invertible, $MM^{-1}=R$, then we can find finitely many homogeneous $a_1, \ldots, a_n\in M$ and $x_1, \ldots, x_n\in M^{-1}$ such that $1=a_1x_1+\cdots+a_nx_n$. Consider the morphisms $\phi_i:M\to R, a\mapsto ax_i$. If $m\in M$, we find that $m=a_1\phi_1(m)+\cdots+a_n\phi_n(m)$, and, by Proposition \ref{DualBasisTheorem}, $M$ is projective. Suppose further that $R$ is Noetherian. Then, $M$ is finitely generated and, by Proposition \ref{P:Eisenbud:11.6-b} i),  there is some fractional superideal $N$ which is isomorphic to $M$.  The invertibility of $M$ implies that $N$ is also invertible. Moreover, by the first part of the proof, both $M$ and $N$ are projective. 
\end{proof}

\begin{proposition}\label{Prop:4.23} Let $R$ be a Noetherian superring and let  $M$ be a projective fractional superideal of $R$. Suppose that $M\ev$ is not contained in $D(K(R)\ev)$. Then, $M$ is invertible. 
\end{proposition}

\begin{proof} Let $\phi_i:M\to R$ and $a_i$, with $i\in I$, as in Proposition \ref{DualBasisTheorem}. Since $M\ev\not\subseteq D(K(R)\ev)$, we may take a fraction $b=u/v\in M\ev-D(K(R)\ev)$ and  $d_i:=b^{-1}\phi_i(b)$. Let $m=s/t\in M$. Since $b, t, v$ are even elements, $u, v, s, t$ belong to $R$ and $M$ is an $R$-supermodule, we find that  

\[
m\phi_i(b)tv=mt\phi_i(bv)=\phi_i(mtbv)=\phi_i(mt)bv=\phi_i(m)btv.
\]

\noindent Thus, in $K(R)$, we have an equation $m\phi_i(b)=b\phi_i(m)$ for all $m\in M$ and, since $b\not\in D(K(R))$, it is invertible and we obtain that   $\phi_i(m)=d_i m$. In other words, $d_i M\subseteq R$ so $d_i\in M^{-1}$ for all $i$. Note that if $m\in M$, then  

\[
m=\sum_{i\in I}d_ia_im,
\]

\noindent where all, but finitely many $b_i$, are zero. If we take $m\not\in  D(K(R))$, we obtain an equation 
 
\[
1=d_{i_1}a_{i_1}+\cdots+d_{i_n}a_{i_n}, \quad \text{ where }a_{i_1}, \ldots, a_{i_n}\in M\text{ and }d_{i_1}, \ldots, d_{i_n}\in M^{-1};
\]

\noindent that is, $1\in MM^{-1}$ and $M^{-1}M=R$, as desired. 
\end{proof}


\section{Dedekind superrings}\label{Sec:Dedekind}


In this section we define and study the properties of Dedekind superdomains.  


\subsection{Integral elements over superrings}\label{Subsec:Integrally}  


Recall that, in this paper, any superring $R$ is supercommutative with unit. Moreover, for any pair of superrings $R$ and $R'$, the inclusion $R\subseteq R'$ means that both have the same multiplicative identity and the $\Z_2$-gradings are compatible, i.e., $R_i\subseteq R'_i$ for all $i\in\Z_2$.

Building on the definition of integral elements for unitary non-commutative rings presented in  \cite{atterton_1972}, we adapt this concept to the realm of superrings. While our definition draws upon the framework of  \cite{atterton_1972}, it diverges to accommodate the unique characteristics of  $\Z_2$-graded rings. This adaptation ensures that the integral closure of a superring maintains its $\Z_2$-graded structure.  Therefore, an element $x$  is considered integral  only if its homogeneous components are themselves integral. Moreover, any element integral under this definition is also integral according to the definition of \cite{atterton_1972}.

Before stating  the definition, we fix some notation. Let $R$ and $R'$ be  superrings such that $R\subseteq R'$. A unitary $R$-supermodule of the form $M=Rb_1+\cdots+ Rb_n$, where $b_1,\ldots,b_n\in R'\ev$, is called \emph{even finitely generated}. Observe that, $M$ is  finitely generated over $R$.  Furthermore, if $M=Rb_1+\cdots+ Rb_n$ and $N=Rc_1+\cdots+ Rc_m$ are unitary even finitely generated $R$-supermodules, then its product $MN=\sum_{i,j}Rb_ic_j=NM$ is a unitary even finitely generated $R$-supermodule.

\begin{definition}\label{Definition:IntegralElements}
Let $R$ and $R'$ be superrings such that $R\subseteq R'$. A homogeneous element $b\in R'$ is called \textit{integral over} $R$ if there exists a unitary even finitely generated $R$-supermodule, $M=Rb_1+\cdots+ Rb_n$, such that $b M\subseteq M$. A non-homogeneous $b\in R'$ is called \textit{integral over} $R$ if its homogeneous components are integral over $R$. 

The set of all integral elements in $R'$   over $R$, denoted by $\cl_{R'}(R)$, is called the \textit{integral closure of $R$ in $R'$}. The superring $R$ is said to be \textit{integrally closed in} $R'$ if $R=\cl_{R'}(R)$. A superring  is called \textit{integrally closed} if it is integrally closed in $K(R)$. In the latter case, we write $\cl(R)$ instead of $\cl_{K(R)}(R)$.
\end{definition}

\begin{proposition}
    Let $R$ and $R'$ be superrings with $R\subseteq R'$. Then, $\cl_{R'}(R)$ is a superring containing $R$.
\end{proposition}

\begin{proof}
Suppose first that $x, y\in \cl_{R'}(R)$ are homogeneous elements. Then, for some positive integers $n$ and $m$, there exist $x_1, \ldots, x_n$, $y_1, \ldots, y_m\in R'{\ev}$ such that the unitary  even finitely generated $R$-supermodules $M=Rx_1+\cdots+R x_n$ and $N=Ry_1+\cdots+Ry_m$ satisfy $xM\subseteq M$ and $yN\subseteq N$. Moreover, $MN$ is a unitary even finitely generated $R$-supermodule  that satisfies $xyMN\subseteq MN$ and $(x+y)MN\subseteq MN$, as one can easily verify. Thus, $xy$ and $x+y$  lie in $\cl_{R'}(R)$. In the general case, let $x$ and $y$ be non-homogeneous in $\cl_{R'}(R)$. Then, by Definition \ref{Definition:IntegralElements}, $x\ev, x\od, y\ev, y\od\in\cl_{R '}(R)$, and by the first part of the proof, $x\ev+y\ev, x\od+y\od, x\ev y\ev+x\od y\od, x\ev y\od+x\od y\ev\in\cl_{R'}(R)$ so $x+y$ and $xy\in\cl_{R'}(R)$. Clearly, $1\in \cl_{R'}(R)$ and therefore $\cl_{R'}(R)$ is an unitary ring. Let $\cl_{R'}(R){\ev}:=\cl_{R'}(R)\cap R\ev'$ and $\cl_{R'}(R){\od}:=\cl_{R'}(R)\cap R\od'$. Thus, $\cl_{R'}(R):=\cl_{R'}(R){\ev}\oplus\cl_{R'}(R){\od}$, endows $\cl_{R'}(R)$ with a $\Z_2$-graduation compatible with the one of $R'$. Finally, to show that $R\subseteq \cl_{R'}(R)$, it suffices to note that if $x\in R$ is a homogeneous element, then the unitary even finitely generated $R$-supermodule  $M:=R1=R$  is such that $xM\subseteq M$. Thus, $\cl_{R'}(R)$ is a superring containing $R$. This completes the proof of the theorem. 
\end{proof}

\begin{proposition}\label{Theorem:2:Att} 
Let $R\subseteq R' \subseteq R''$ be superrings. If $R''=\cl_{R''}(R')$ and  $R'=\cl_{R'}(R)$, then $R''=\cl_{R''}(R)$. In particular, $\cl_{R'}(R)$ is integrally closed in $R'$.
\end{proposition}

\begin{proof}
    The proof follows the same arguments as in \cite[Theorem 2]{atterton_1972}.
\end{proof}

Let $R$ and $R'$ be commutative rings with $R\subseteq R'$. Recall that, in commutative algebra, an element $x\in R'$ is \textit{integral over} $R$ if $x$ is a root of a monic polynomial with coefficients in $R$ (see \cite[p. 59]{atiyah}). While this definition proves valid for commutative rings, its applicability falters in the realm of  non-commutative rings. However, a weaker version is provided for the case of superring, by the following proposition.

\begin{proposition}\label{Att:P2}
Let $R\subseteq R'$ 
 be superrings. If $b\in R'$ is integral over $R$ and $b$ commutes with every element of 
$R$, then $b$ satisfies the equation
\begin{equation}\label{eq:AttP2}
b^n+a_1b^{n-1}+\cdots+a_{n-1}b+a_n=0, \quad  \text{where} \quad a_1, \ldots, a_n\in R.
\end{equation}
\noindent Conversely, if $b\in R'{\ev}$ satisfies an equation of the form  \eqref{eq:AttP2}, then $b\in \cl_{R'}(R)$.   
\end{proposition}

\begin{proof} The proof is a very close adaptation of the proof of Property 3 in \cite[\S 1]{atterton_1972}.
\end{proof}

\begin{proposition}\label{xpnotinp}
 Let $R$ be a Noetherian integrally closed superdomain that is not a superfield. If $x\in K(R)$ is an even element with $x\not\in R$, then any nonzero regular prime ideal $\p$ of $R$ satisfies that $x\,\p\not\subseteq\p$. In particular, if $R$ is a strong superdomain, then any nonzero prime ideal $\p$ of $R$, other than $\J_R$, satisfies that $x\,\p\not\subseteq\p$.
\end{proposition}

\begin{proof}
Suppose, for a contradiction, that $x\,\p\subseteq\p$.  Since $\p\not\subseteq D(R)$, then we can take $a\in\p\ev- D(R)\ev$. By hypothesis, $x\p\subseteq\p$,  hence $x^ia\in R$ for any integer $i>0$. Then, defining the superideals $$\a_i:=(a, x^1a, x^2a, \ldots, x^ia)_R,$$ for each integer $i>0$, we obtain an ascending chain of superideals of $R$, given by   $$\a_1\subseteq\a_2\subseteq\cdots\subseteq \a_i\subseteq\a_{i+1}\subseteq\cdots.$$

Since $R$ is Noetherian,  $\a_{i+d}=\a_d$, for some $d$ and every $i$. This implies that $$x^{d+1}a=b_0a+b_1x^1a+b_2x^2a+\cdots+b_dx^da,$$ where $b_0, \ldots, b_d\in R.$  Because $a\not\in D(R)$ is invertible in $K(R)$,  we get the equation $$x^{d+1}=b_0+b_1x^1+b_2x^2+\cdots+b_dx^d.$$ But, by Proposition \ref{Att:P2}, $x$ is integral over $R$, contradicting the hypothesis of $R$ being integrally closed.
\end{proof}

\begin{remark}
    Proposition \ref{xpnotinp}  remains valid even when employing as definition of integrally closed superring  the one proposed for non-commutative rings in \cite{atterton_1972}. This is due to the fact that the core of its proof relies intrinsically on    Proposition \ref{Att:P2} and the property of Noetherianity. 
\end{remark}

The following proposition will be used to prove Proposition \ref{lem:Dedinv}.

\begin{proposition}\label{lem:facprime}
Let $R$ be a Noetherian superring and $\a$ a nonzero proper superideal of $R$ that is not prime. Then, there are finitely many prime ideals $\p_1, \ldots, \p_d$ of $R$ such that $\p_1\cdot\p_2\cdots\p_d\subseteq\a\subseteq\p_1\cap\p_2\cap\cdots\cap\p_d$.
\end{proposition}
 
\begin{proof}
The proof is analogous to that of the commutative case (see \cite[Lemma 1.9]{narkiewicz2004elementary}).
\end{proof}

\begin{proposition}\label{lem:Dedinv}
Let $R$ be a Noetherian integrally closed superdomain that is not a superfield, such that every prime ideal of $R$, other than $\J_R$, is maximal. If $\p$ is a regular prime ideal, then $\p$ is invertible. In particular, if $R$ is a strong superdomain and $\p\neq\J_R$, is a prime ideal of $R$, then $\p$ is invertible. 
\end{proposition}
\begin{proof}
For $a\in\p\ev-D(R)\ev$, by Proposition \ref{lem:facprime}, there exist a minimal set of prime ideals of $R$, $\p_1, \ldots, \p_s$, say, such that  $\p_1\cdots\p_s\subseteq aR\subseteq\p_1\cap\cdots\cap \p_s.$ In particular, we find that $\p_1\cdots\p_s\subseteq\p$, because $aR\subseteq\p$. Now, by \cite[Theorem 4.1.3]{westrathesis}, it follows that for some $i$, $\p_i\subseteq \p$ and thus $\p=\p_i$, since we have supposed that every prime ideal is maximal. We may assume, without loss of generality, that $\p=\p_1$. By the minimality of the $\p_i$, the product $\p_2\cdots\p_s$ cannot be contained in $aR$. Hence, there exists a homogeneous element $b$ in the latter product not  in $aR$. Now, if we define $x:=a^{-1}b$, then for any $y\in\p$, we find 
 that $by\in\p_1\cdots\p_s$ and, consequently, $by\in aR$. Therefore, $xy=a^{-1}by\in R$, i.e., $x\,\p\subseteq R$. Observe that, $x\in\p^{-1}$ and $x\not\in R$. Indeed, if $x\in R$, $b=ax\in\p=\p_1$, which is impossible. Since $\p^{-1}\p$ is clearly a superideal of $R$ containing $\p$,  by the  maximality of $\p$, it follows that there are only two possible cases: $\p^{-1}\p=R$ or $\p^{-1}\p=\p$. If the latter case occurs, $x\,\p\subseteq\p^{-1}\p\subseteq\p$ which contradicts Proposition \ref{xpnotinp}. Thus,  $\p$ is  an invertible superideal of $R$.
\end{proof}

\subsection{Regular superrings with even Krull super-dimension one}\label{Subsec:Dedekind:Is:Regular}\textit{ }


\begin{definition}
    A \textit{discrete valuation superring} is a regular local Noetherian superring of Krull superdimension $1\mid N$, where $N$ is a non-negative integer. 
\end{definition}

\begin{theorem}\label{Theorem:5.9}
      Let $(R, \m)$ be a Noetherian local  superring, where $\m$ is regular and let $\Bbbk=R/\m=R\ev/\m\ev$ be the residue field. The following conditions are equivalent. 
     
     \begin{itemize} 
         \item[\rm i)] $R$ is a discrete valuation superring. 
         \item[\rm ii)] \begin{itemize} 
             \item[\rm a)] $\overline{R}$ is a discrete valuation ring. 
             \item[\rm b)]  For any/some odd elements $z_1, \ldots, z_s$ of $\m$ that give rise to a $\K$-basis of $(\m/\m^2)\od$, we have $z^{[s]}\neq 0$. 
         \end{itemize}
     \end{itemize}   
\end{theorem}

\begin{proof} i) $\Rightarrow$ ii): By i) and  \cite[Proposition 5.2]{MASUOKA2020106245}, $\overline{R}$ is regular. On the other hand, since $\overline{R}$ is a domain, we find that $\nil(R\ev)=R\od^2$, so $\kdim(\overline{R})=\kdim(R\ev)=1$. Thus, $\overline{R}$ is a DVR and ii) a) is proved.  

    Since $R$ is regular, $\ksdim\od(R)=\dim_\Bbbk((\m/\m^2)\od)=s$. Let $\{\overline{z}_1, \ldots, \overline{z}_s\}\subseteq(\m/\m^2)\od$ be a basis for the $\Bbbk$-vector space $(\m/\m^2)\od=R\od/\m\ev R\od$. By Nakayama's Lemma, $z_1, \ldots, z_s$ generate the $R\ev\,$-module $R\od$. By Proposition \cite[Proposition 4.1]{MASUOKA2020106245}  and $\ksdim\od(R)=s$, there is a system of odd parameters of $R$ among the $z_i$ whose length is $s$. That is, $z_1, \ldots, z_s$ is a system of odd parameters, \textit{a fortiori}, $z^{[s]}\neq0$, and b) is proved.  

 ii) $\Rightarrow$ i): Since $(R, \m)$ is local,  $(\overline{R}, \overline{\m})$ is local. By ii) a), $\overline{R}$ is DVR,     $1=\kdim(\overline{R})=\kdim(R\ev)=\ksdim\ev(R)$ and

    \[
    1=\kdim(\overline{R})=\dim_{\overline{R}/\overline{\m}}\left(\overline{\m}/\overline{\m}^2\right)=\dim_{R/\m}\left((\m/\m^2)\ev\right).
    \]

    \noindent It remains to prove that $\ksdim\od(R)=\dim_{\Bbbk}((\m/\m^2)\od)$. Let $s=\dim_{\Bbbk}((\m/\m^2)\od)$. Then, $\ksdim\od(R)\leq s$. We may assume that $s\geq 1$, because if $s=0$, then $\ksdim\od(R)=0$ so $R\od=0$ and  there is nothing to prove. Now, let $\{\overline{z}_1, \ldots, \overline{z}_s\}\subseteq(\m/\m^2)\od$ be a basis for the $\Bbbk$-vector space $(\m/\m^2)\od$ such that $z^{[s]}\neq0$. We have two cases to consider, namely:

    \noindent \underline{Case 1:} $\kdim(R\ev/\ann_{R\ev}(z^{[s]}))=\kdim(R\ev)$.
    
    In this case, $z_1, \ldots, z_s$ is a system of odd parameters and it follows that $\ksdim\od(R)=s=\dim_{R/\m}((\m/\m^2)\od)$. Hence, $R$ is regular. 
    \\

    \noindent \underline{Case 2:}  $\kdim(R\ev/\ann_{R\ev}(z^{[s]}))<\kdim(R\ev)=1$.

    In this case, $\kdim(R\ev/\ann_{R\ev}(z^{[s]}))=0$, so the only prime ideal in $R\ev/\ann_{R\ev}(z^{[s]})$ is the maximal one, namely $\m\ev/\ann_{R\ev}(z^{[s]})$. Hence, the quotient should be Artinian and we obtain $(\m\ev/\ann_{R\ev}(z^{[s]}))^n=0$, for some positive integer $n$. Thus,
 
\begin{align*}
    0&=(\m\ev/\ann_{R\ev}(z^{[s]}))^n 
     =\left(\m\ev^n+\ann_{R\ev}(z^{[s]})\right)/\ann_{R\ev}(z^{[s]}) 
     \simeq\m\ev^n/\left(\m\ev^n\cap \ann_{R\ev}(z^{[s]})\right).
\end{align*}

\noindent It follows that 
\begin{align*}
    \m\ev^n \subseteq\m\ev^n\cap \ann_{R\ev}(z^{[s]}) 
     \subseteq\ann_{R\ev}(z^{[s]}).
\end{align*}

The above inclusion implies that any element of $\m\ev$ to the power of $n$ is in $\ann_{R\ev}(z^{[s]})$. Consider some (non-zerodivisor) $m\in\m\ev-R\od^2$. Then, we obtain that $m^n\in\ann_{R\ev}(z^{[s]})$, i.e., $m^n z^{[s]}=0$. Since $m$ is not a zerodivisor, $m^{n-1}z^{[s]}=0$ and, through a recursive reasoning, it is not hard to see that $z^{[s]}=0$, a contradiction. This contradiction shows that  $\kdim(R\ev/\ann_{R\ev}(z^{[s]}))=\kdim(R\ev)=1,$  which means that $R$ is regular.  
\end{proof}

\begin{remark} With the assumptions of  Theorem \ref{Theorem:5.9}, i) or/and ii) imply  

\begin{itemize}
         \item[iii)] \begin{itemize}
             \item[\rm a)] $\overline{R}$ is a discrete valuation ring.
             \item[\rm b)]   For some/any minimal system $z_1, \ldots, z_s$ of generators of the $R\ev$-module $R\od$, $z^{[s]}\neq 0$. 
         \end{itemize}
     \end{itemize}  

 It would be interesting to know if iii) $\Rightarrow$ i). Furthermore, does Theorem \ref{Theorem:5.9} hold without the assumption of $\m$ being regular? 
\end{remark}

\begin{proposition}\label{Corollaty:to:Theorem:5.9}
   Let $R$ be a local Noetherian superdomain which is not a superfield, with regular maximal ideal  such that:

    \begin{itemize} 
        \item[\rm i)] $R$ is integrally closed. 
        \item[\rm ii)] Every prime superideal of $R$, other than $\J_R$, is maximal. 
        \item[\rm iii)] For some/any basis $\overline{z}_1, \ldots, \overline{z}_s$ of the $\Bbbk$-vector space $V=(\m/\m^2)\od$, we have $z^{[s]}\neq 0$.
    \end{itemize}

    \noindent Then, $R$ is regular with $\ksdim(R)=1\mid N$, where $N$ is a non-negative integer. That is, $R$ is a discrete valuation superring.
\end{proposition}

\begin{proof} Suppose that $R$ is a local superring, with maximal ideal $\m$.  Since $R$ is Noetherian, $\overline{R}$ is Noetherian as well (see, e.g., \cite[Corollary 3.3.4]{westrathesis}). Clearly, $\overline{R}$ is also a local domain whose maximal ideal is $\overline{\m}=\m/\J_R$. We claim that $\overline{\m}$ is invertible. To see this, note that for any prime ideal $\overline{\p}$ of $\overline{R}$,  $\overline{\m}_{\overline{\p}}\simeq\overline{\m_\p}\simeq\overline{R_\p}\simeq\overline{R}_{\overline{\p}},$ where we have used  \cite[Proposition 5.1.4]{westrathesis} and Proposition \ref{lem:Dedinv}. Now, we prove that $\overline{R}$ is discrete valuation ring. It suffices to show that any nonzero ideal of $\overline{R}$ is a power of $\overline{\m}$ (see  \cite[Proposition 9.7]{atiyah}). Suppose, for a contradiction, that the assertion is not true and consider the collection $\mathcal{S}$ of nonzero ideals of $\overline{R}$ that are not a power of $\overline{\m}$, ordered by inclusion. By our hypothesis, this collection is nonempty and, since $\overline{R}$ is Noetherian, $\mathcal{S}$ has a maximal element, say $\a$. Clearly, $\overline{\m}\not\in\mathcal{S}$ and thus $\a\subseteq\overline{\m}$, because $\a\neq\overline{\m}$.

Assume, for a contradiction, that $\overline{\m}^{-1}\a=\overline{R}$. Then, there are finitely many $a_i\in \overline{\m}^{-1}$ and $x_i\in \a$, with $i=1, \ldots, n$, such that $1=a_1x_1+\cdots+a_nx_n$. If for all $i=1, \ldots, n$, $a_ix_i\in\overline{\m}$, then $1\in\overline{\m}$, a contradiction. Thus, for some $j$, $a_jx_j\not\in\overline{\m}$. Let us fix a such $j$. Since $(\overline{R}, \overline{\m})$ is local, $a_jx_j$ must be a unit. Further, any $m\in\overline{\m}$ may be written in the form  $m=[(ma_j)(a_jx_j)^{-1}]x_j.$ That is, $\overline{\m}=x_j\overline{R}\subseteq\a$ (since $x_j\in\a$), i.e., $\overline{\m}=\a$, contradicting that $\a$ is properly contained in $\overline{\m}$. Therefore,  $\overline{\m}^{-1}\a\subsetneq\overline{R}=\overline{\m}^{-1}\overline{\m}.$  Thus, $ \a\subseteq \overline{\m}^{-1}\a\subsetneq\overline{\m}^{-1}\overline{\m}=\overline{R},$   since $\overline{\m}$ is invertible. Thus, $\overline{\m}^{-1}\a$ is a proper ideal of $\overline{R}$. If $\overline{\m}^{-1}\a=\a$, then $\a=\overline{\m}\a$, so $\a=0$ by (the usual) Nakayama's Lemma. Hence, $\overline{\m}^{-1}\a\supsetneq\a$. Therefore, the maximality of $\a$ implies that $\overline{\m}^{-1}\a$ is a power of $\overline{\m}$, say $\overline{\m}^{-1}\a=\overline{\m}\,^{r}$. We conclude that $\a=\overline{\m}\,^{r+1}$, i.e., $\a$ is a power of $\overline{\m}$, a contradiction. Thus, $\overline{R}$ is a DVR. The subsequent demonstration proceeds analogously to that of Theorem \ref{Theorem:5.9}.
\end{proof}

\begin{remark}\label{remark:5.12}\textit{ }
 
\begin{enumerate}[i)]
    \item  While Proposition \ref{Corollaty:to:Theorem:5.9} is to be expected, in view of the usual theory of discrete valuation rings and the results we have obtained throughout this paper, it has surprising and unexpected consequences in the realm of superrings. In fact, let $R$ be a local superring as in Proposition \ref{Corollaty:to:Theorem:5.9}. Thus, $\m$, the maximal ideal of $R$, is invertible by Proposition \ref{lem:Dedinv}. This implies that $\m=x R, \text{ for some } x\in\m-\J_R$, by Proposition \ref{LEMA:3.11}. In other words, $\m$ has one generator and hence the total dimension of $R$ is one. Since $R$ is regular, the total dimension of $R$ coincides with the number $\ksdim\ev(R)+\ksdim\od(R). $ Since, $\ksdim\ev(R)=1$, then $\ksdim\od(R)=0$. That is, $\J_R$ has no generators as a superideal of $R$, so $\J_R=0$ and $R\od=0$. Hence, $R=R\ev=\overline{R}$.
    \item By i), if $R$ is a discrete valuation  superdomain with nontrivial odd part and regular maximal ideal, then it cannot be integrally closed and the maximal ideal cannot be invertible (and hence not projective by Proposition \ref{Prop:4.23})! Further, an integrally closed  superdomain with Krull superdimension $1\mid N$ and regular maximal ideal cannot be regular! 

\item  In classical commutative algebra, the integral closure of a Dedekind domain within a finite separable field extension remains a Dedekind domain. However, our investigation demonstrates that no such stable formulation exists in the supercommutative setting; the structural constraints of $\mathbb{Z}_2$-graded extensions prevent a direct analogue of this property.

\item For classical Dedekind domains, the notions of Principal Ideal Domain and Unique Factorization Domain are equivalent. In \cite{RTT5}, we established that any superdomain possessing a non-trivial odd part and satisfying unique factorization must necessarily have an even Krull superdimension zero. Consequently, the traditional equivalence between principal ideal domains and unique factorization domains fails to generalize.

\item Our analysis reveals that, excluding the trivial purely even case, 
a discrete valuation superring cannot arise from a discrete valuation on its associated superfield of fractions. While any superring derived from such a valuation is inherently integrally closed, we establish in the present work that non-trivial Dedekind superrings are never integrally closed. This highlights a sharp distinction between the $\mathbb Z_2$-graded and non-graded valuation theories. For a systematic treatment of valuation theory in the supercommutative context, we refer the reader to \cite{RTT3}.

\end{enumerate}
   
\end{remark} 

\begin{proposition}\label{cor:sdvr} Let $R$ be a Noetherian superdomain. The following conditions are equivalent. 

\begin{itemize}
    \item[\rm i)] $R$ is regular with Krull superdimension $1\mid N$. 
    \item[\rm ii)] $R_\p$ is a discrete valuation superring for all prime ideal $\p$ of $R$.
\end{itemize}

\end{proposition}

\begin{proof}  i) $\Rightarrow$ ii) If $R$ is regular, then for any prime ideal $\p$ of $R$, the local superring $R_\p$ is local by Definition \ref{def:regular:superring}. Now, we will compute $\ksdim\ev(R_\p)$. Recall that if $A$ is a commutative ring, then $\kdim(A/\nil(A))=\kdim(A)$ and since $R$ is a superdomain, then $\nil(R\ev)=R\od^2$. Thus, we find that 

    \begin{align*}
        \ksdim\ev(R_\p) =\kdim((R_\p)\ev) 
         =\kdim((R\ev)_{\p\ev}) 
         =\kdim((R\ev)_{\p\ev}/\nil((R\ev)_{\p\ev})) 
         =\kdim(\overline{R}_{\overline{\p}})
         =1.
    \end{align*} 

    \noindent Thus, $R_\p$ is a discrete valuation superring for all prime $\p$ and ii) holds. 
    
 ii) $\Rightarrow$ i) If every $R_\p$ is regular, then $R$ is regular by Definition \ref{def:regular:superring}. On the other hand, note that for all prime $\overline{\p}$ of $\overline{R}$, the ring $\overline{R_\p}=\overline{R}_{\overline{\p}}$ is a DVR. Thus, $\overline{R}$ is a Dedekind ring and $\ksdim\ev(R)=\kdim(R\ev)=\kdim(\overline{R})=1$. Hence, i) holds and the proof is complete. 
\end{proof}

\begin{definition}
    A superdomain $R$ is called a \textit{Dedekind superring} if it satisfies one (and hence all) of the equivalent conditions of Proposition \ref{cor:sdvr}. 
\end{definition}

\begin{example}
    Let $\mathbb{C}$ be the field of complex numbers. Then, for every non negative integer $N$, the superring  $R=\mathbb{C}[\,x\mid\theta_1, \,\theta_2, \,\ldots, \,\theta_N]$ is a Dedekind superring with Krull superdimension $1\mid N$. In particular, consider $N=1$, i.e., $R=\mathbb{C}[\,x\mid\theta\,]$. Let $\m=(x-\alpha\mid\theta)$ be a maximal ideal of $R$. Note that $\overline{R}=\mathbb{C}[x]=R\ev$ and $\m\ev=(x-\alpha)_{\overline{R}}$. Let $\p\neq\m$ be any prime ideal of $R$. Then, $\m_\p\simeq R_\p$ as $R_\p$-supermodules. However, if $\p=\m$, the isomorphism is of $(R\ev)_{\p\ev}$-supermodules, but not as $R_\p$-supermodules. Hence, $\m$ is not invertible and $R$ is not integrally closed (by Proposition \ref{lem:Dedinv}). \qed
\end{example}

\begin{remark}
    Two key properties of  Dedekind rings (integrally closed and invertibility of ideals) have just been revealed not to be superizable. As we will see, the list extends even longer. For instance, the prime factorization of ideals fails to hold when $R\od\neq0$. Let $R$ be a discrete valuation superring with maximal ideal $\m$. Suppose that every superideal of $R$ is a power of $\m$. If $\m^2=\m$, we use the Super Nakayama's Lemma to find that $\m=0$. Thus, $\m^2\neq\m$. Let us consider a homogeneous element $x\in\m-\m^2$. Then, $(x)_R=\m^s$ for some $s\geq0$. Since $x\not\in\m^2$, the only possibility is that $s=1$ and thus $\m$ is principal. We have already shown above that this implies that $R$ has trivial superring structure (see Remark \ref{remark:5.12}).   However, we do have primary factorization for superideals, as consequence of a more general result. Namely, using the results on primary ideals in \cite[\S\S 4.2-4.3]{westrathesis} and following the usual treatment for the commutative counterpart of this topic (e.g., \cite[Proposition 9.1]{atiyah}), one can prove that if $R$ is a Noetherian superdomain with Krull superdimension $1\mid N$, then every non-zero superideal of $R$ can be uniquely expressed as the product of primary superideals whose radicals are all distinct. The details of this assertion can be found in \cite[Chapter 4]{ThesisJT}. 
\end{remark}


\subsection{Dedekind superschemes}\label{SUP:SCH}

For the necessary definitions of locally ringed superspaces and superschemes, we refer to \cite{MASUOKA2020106245} or \cite{westrathesis}.

\begin{definition}
    A \textit{Dedekind superscheme} $ \mathfrak{X}=(X, \O_X)$  is a superscheme admitting a finite affine covering $$X=\bigcup_{i\in I}\spec(R_i),$$ such that each $R_i$ is a Dedekind superring. 
\end{definition}

\begin{example} Let $\mathfrak{X}=(X, \O_X)$ be a  superscheme of finite type over a field $\K$. Suppose $\mathfrak{X}$ is nonsingular and has superdimension $1\mid N$. Thus, for any open affine $U\subseteq X$, the superring $\O_X(U)$ is a Dedekind superring and for any prime $\p$, the stalk $\O_{X,\p}$ is a discrete valuation superring.   
\end{example}

\begin{example}
    If $\mathfrak{X}$ is a Dedekind superscheme, it follows that it is Noetherian, integral, nonsingular, and has superdimension $1\mid N$. Every Dedekind scheme can be naturally interpreted as a Dedekind superscheme, where the underlying Dedekind superrings simply have trivial odd parts. Furthermore, the spectrum of a Dedekind superring is itself a Dedekind superscheme. To illustrate this last point, consider an algebraically closed field $\K$. We can then construct the superpolynomial ring  $R=\K[\,x\mid\theta_1,\, \ldots, \,\theta_N]\,$  which becomes a Dedekind superdomain with Krull superdimension $1\mid N$. Its spectrum,  $\spec(R)=\mathbb{A}^{1|N}$, forms an affine Dedekind superscheme. Using standard gluing techniques, we establish that projective lines $\mathbb{P}^{1|N}_\K$ also belong to the class of Dedekind superschemes.
\end{example}

Recall that the topological space $Y$ of any Dedekind scheme possesses at least two distinct points. One is the generic point $\eta$ (such that $\overline{\{\eta\}}=Y$) and the other points are all closed. Now let $\X$ be a Dedekind superscheme and  $\X_{\mathrm{ev}}$ is bosonic reduction (which is a Dedekind scheme). Because $\X$ and its bosonic reduction $\X_{\mathrm{ev}}$ share the same underlying topological space, the topological space of $\X$ has one generic point and one or more closed points. By the same argument, we can demonstrate the following proposition.

\begin{proposition}\label{Prop:5.9}
    Let $\X=(X, \O_X)$ be a Dedekind superscheme. Then, every non-empty open subsets of $X$ has the form $X-\{x_1, \ldots, x_k\}$, for finitely many closed points $x_1, \ldots, x_k\in X$. In particular, the set $\{\eta\}$, consisting of the generic point $\eta$, is open if and only if $X$ is finite.  
\end{proposition}

The following proposition characterizes the structure sheaf of a Dedekind superscheme. 

\begin{proposition} Let $\X=(X, \O_X)$ be a Dedekind superscheme with generic point $\eta$. Then, the stalk $\O_{X, \eta}$ is a superfield. Furthermore, for any non-empty open affine subset of $X$, $Q(\O_X(U))\simeq\O_{X, \eta}$. Moreover, if $x\in X$ is any closed point, $\O_{X,  x}$ is a discrete valuation superring whose superfield of fractions is $\O_{X, \eta}$. 
\end{proposition}

\begin{proof}
    By \cite[Lemma 2.3]{MASUOKA2020106245},  $\O_{X, \eta}$ is a superfield and $Q(\O_{X}(U))\simeq\O_{X, \eta}$. Now,  let $x\in X$ be a closed point. We then consider a non-empty affine open neighbourhood $U$ of $x$. Then, $\O_X(U)$ is a Dedekind superring. In addition, the closed point $x\in X$ corresponds to a  maximal ideal $\m_x\in\O_{X}(U)$, while the generic point necessarily corresponds to the canonical superideal $\J_{\O_X(U)}$. Consequently, we have $\O_{X}(U)_{\m_x}\simeq\O_{X, x}$ and $Q(\O_{X, x})\simeq \O_{X, \eta}$. \end{proof}

\begin{proposition}\label{P:6.8}
    Any open subsuperscheme of an affine Dedekind superscheme is affine. 
\end{proposition}

\begin{proof}
    Let $\mathfrak{X}$ be an affine Dedekind superscheme and consider $\mathfrak{Y}$ any open subsuperscheme of $\mathfrak{X}$. Then, $\mathfrak{X}_{\mathrm{ev}}$ is a Dedekind scheme and further $\mathfrak{Y}_{\mathrm{ev}}$ is an open subscheme of $\mathfrak{X}_{\mathrm{ev}}$. Thus, $\mathfrak{Y}_{\mathrm{ev}}$ is affine (see \cite[Lemma 15.17]{gortz}). Therefore, since $\mathfrak{Y}$ is clearly Noetherian, it is affine by \cite[Theorem 3.1]{ZubkovNoetherian}.
\end{proof}

\begin{proposition}
    Let $\mathfrak{X}$ be an affine Dedekind superscheme. Then, for any open subscheme of $\mathfrak{X}$, say $\mathfrak{Y}$, we have 
    $$H^i(\mathfrak{Y}, \mathcal{F})=0\quad\text{  for all $i\geq 1$}$$ and all quasicoherent $\mathcal{F}$ of $\mathfrak{Y}$-supermodules.
\end{proposition}

\begin{proof}
    Any such $\mathfrak{Y}$ is affine by Proposition \ref{P:6.8}. Then, the result follows from \cite[Theorem 3.2]{ZubkovNoetherian}.
\end{proof}

We are now equipped with the requisite tools to tackle the following two key problems. 
\medskip

\noindent \textbf{Problem 1.} Develop a theory of Dedekind superschemes.
\medskip

In the classical commutative setting, every coherent sheaf of ideals $\mathcal{I}$ on a Dedekind scheme $X$ is invertible. This property is established stalk-wise: for each point $p \in X$, the $\mathcal{O}_{X,p}$-module $\mathcal{I}_p$ is free of rank one. While this freeness is trivial at the generic point, it follows at closed points from the fact that $\mathcal{O}_{X,p}$ is a discrete valuation ring, and consequently a Principal Ideal Domain. As previously noted, these fundamental algebraic characterizations fail for discrete valuation superrings. Consequently, the natural super-analogue of the aforementioned invertibility does not hold in general. This divergence suggests that the theory of Dedekind superschemes will exhibit significant structural discrepancies compared to the classical case. In particular, this setting necessitates a dedicated development of a theory of Weil divisors specifically adapted to the super-geometric context.

\medskip

\noindent \textbf{Problem 2.} Let $\mathfrak{S}$ be a Dedekind superscheme. We call an integral, projective, flat $\mathfrak{S}$-superscheme $\pi: \mathfrak{X} \to \mathfrak{S}$ of dimension $2\mid N$ a \textit{fibered supersurface} over $\mathfrak{S}$. Develop a theory of fiber supersurfaces. 

\section{Final Comments}\label{comments}


We conclude by outlining several directions for future research.  

Our investigation of valuations on superrings revealed striking discrepancies between the commutative and supercommutative settings. In particular, we proved that every valuation superring is integrally closed, which detaches this notion from one-dimensional regularity. Consequently, the classical correspondence between valuation rings and Dedekind domains no longer holds in the super context. These differences, together with a general framework for Manis valuations on superrings, are analyzed in detail in our paper \cite{RTT3}.  

Another open direction concerns the classification of finitely generated supermodules over Dedekind superrings. The absence of prime factorization for superideals prevents a direct extension of the classical Dedekind-domain methods, such as representing ideals as $\a=(a,b)_R$.  

Finally, for a Noetherian superring $R$, let $\spic_-(R)$ denote the set of isomorphism classes of locally free $R$-supermodules of rank $0\mid1$. We expect $\spic_-(R)$ to carry a group structure analogous to $\spic_+(R)$, motivating the study of the \textit{total Picard group} $\spic(R)=\spic_+(R)\amalg\spic_-(R)$ and the characterization of $\ker(\pi_*)$ in this setting.


\section*{Acknowledgements}

The authors are especially grateful to Alexandr Zubkov (Department of Mathematical Sciences, UAEU, Al-Ain, United Arab Emirates; Sobolev Institute of Mathematics, Russia) for the very useful comments, suggestions and several enriching discussions about several points in this work.

\subsection*{Funding}
The first and second authors were partially supported by CODI (Universidad de Antioquia, UdeA) by project numbers 2020-33713 and 2022-52654.


\end{document}